\documentclass[10pt,a4paper]{article}
\usepackage{color}
\usepackage{amsmath}
\usepackage{amssymb}
\usepackage{amsthm}
\usepackage{graphicx}
\usepackage{url}
\usepackage{algorithm}
\usepackage{algorithmic}

\numberwithin{equation}{section}

\newcommand{\Z}{\mbox{$\mathbb{Z}$}}
\newcommand{\wt}[1]{\widetilde{#1}}
\newcommand{\dsm}[1]{\mbox{$\displaystyle #1 $}}
\newcommand{\Primes}{{\cal P}}
\newcommand{\qtx}[1]{\quad\text{#1}\quad}
\newcommand{\setof}[2]{\left\{#1\, \dv #2\right\}}
\def\qh{\widehat{q}}
\def\Rh{\widehat{R}}
\def\Fh{\widehat{F}}
\def\Pih{\widehat{\Pi}}
\def\pih{\widehat{\pi}}

\def\ds{\displaystyle}
\def\iy{\infty}

\def\ds{\displaystyle}
\def\iy{\infty}

\def\ds{\displaystyle}
\def\iy{\infty}

\def\al{\alpha}
\def\be{\beta}
\def\ga{\gamma}
\def\de{\delta}
\def\De{\Delta}
\def\la{\lambda}
\def\La{\Lambda}

\def\vep{\varepsilon}
\def\vfi{\varphi}
\def\om{\omega}

\def\si{\sigma}

\def\cd{\mathcal{D}}
\def\ce{\mathcal{E}}

\def\cg{\mathcal{G}}

\def\cn{\mathcal{N}}
\def\co{\mathcal{O}}
\def\cp{\mathcal{P}}
\def\cs{\mathcal{S}}
\def\ct{\mathcal{T}}
\def\cu{\mathcal{U}}

\def\ni{\noindent}

\def\dv{\,|\,}

\def\Li{\textrm{Li}}
\def\Llr{\Longleftrightarrow}

\def\LR{\Longrightarrow}
\def\numero{n$^{\text{o}}$}

\newcommand{\fl}[1]{\left\lfloor {#1}\right\rfloor}

\newcommand{\pd}[2]{\frac{\partial#1}{\partial#2}}

\newcommand{\set}[1]{\left\{#1\right\}}

\newcommand{\ben}{{\rm ben\,}}
\newcommand{\benp}{{\rm ben}_p\,}
\newcommand{\sg}{''}
\newcommand{\pfrac}[2]{\left(\frac{#1}{#2}\right)}

\newtheorem{definition}{Definition}
\newtheorem{prop}{Proposition}

\newtheorem{lem}{Lemma}
\newtheorem{coro}{Corollary}
\newtheorem{principe}{Principle}
\newenvironment{rem}{\noindent {\bf Remark}: }%
{\smallskip\par}

\title{Landau's function for one million billions}
\author{Marc Del\'eglise, Jean-Louis Nicolas and Paul Zimmermann
\footnote{Research partially supported by INRIA and by CNRS.}}

\begin{document}
\maketitle

\hfill
\begin{minipage}[t]{40mm}
\`A Henri Cohen pour son soixanti\`eme 
anniversaire.
\end{minipage}
\bigskip

\begin{abstract}
Let ${\mathfrak S}_n$ denote the symmetric group 
with $n$ letters, and $g(n)$ the maximal order of an element 
of ${\mathfrak S}_n$. If the standard factorization of $M$ into primes 
is $M=q_1^{\al_1}q_2^{\al_2}\ldots q_k^{\al_k}$, we define $\ell(M)$
to be $q_1^{\al_1}+q_2^{\al_2}+\ldots +q_k^{\al_k}$;
one century ago, E. Landau proved that $g(n)=\max_{\ell(M)\le n} M$ and that,
when $n$ goes to infinity, $\log g(n) \sim \sqrt{n\log(n)}$.

There exists a basic algorithm to compute $g(n)$ for $1 \le n 
\le N$; its running time is 
$\co\left(N^{3/2}/\sqrt{\log N}\right)$
and the needed memory is $\co(N)$; 
it allows computing $g(n)$ up to, say, one million. We describe an 
algorithm to calculate $g(n)$ for $n$ up to $10^{15}$. The main idea is to use 
the so-called {\it  $\ell$-superchampion numbers}. Similar numbers, the 
{\it superior highly composite numbers}, were introduced by S. Ramanujan to 
study large values of the divisor function $\tau(n)=\sum_{d\dv n} 1$.
\end{abstract}

\bigskip

\ni

\bigskip

\ni

%--------------------------------------------------------------------------
\section{Introduction}
%--------------------------------------------------------------------------
\subsection{Known results about Landau's function}
%--------------------------------------------------------------------------

For $n\ge 1$, let ${\mathfrak S}_n$ denote the symmetric group 
with $n$ letters. The order
of a permutation of ${\mathfrak S}_n$ is the least common 
multiple of the lengths of its cycles.
Let us call $g(n)$ the maximal order of an element of ${\mathfrak S}_n$. 
If the standard factorization of $M$ into primes 
is $M=q_1^{\al_1}q_2^{\al_2}\ldots q_k^{\al_k}$, we define $\ell(M)$
to be 
\begin{equation}\label{l}
\ell(M)=q_1^{\al_1}+q_2^{\al_2}+\ldots +q_k^{\al_k}.
\end{equation}
E. Landau proved in \cite{Lan} that
\begin{equation}\label{g}
g(n)=\max_{\ell(M)\le n} M
\end{equation}
which implies 
\begin{equation}\label{lgn}
\ell(g(n))\le n
\end{equation}
and for all positive integers $n,M$
\begin{equation}\label{pcarg}
\ell(M) \le n \LR  M \le g(n)\quad \iff \quad M > g(n) \LR \ell(M) > n.
\end{equation}
P. Erd\H {o}s and P. Tur\'an proved in \cite{Tur} that 
\begin{equation}\label{ErTu}
M \text{ is the order of some element of } {\mathfrak S}_n \;\Llr\; 
\ell(M)\le n.
\end{equation} 
E. Landau also proved in \cite{Lan} that
\begin{equation}\label{gasy}
\log g(n) \sim \sqrt{n\log n}, \qquad n\to \iy.
\end{equation}
This asymptotic estimate was improved by S. M. Shah \cite{Sha} and M. Szalay
\cite{Sza}; in \cite{MNR1}, it is shown that
\begin{equation}\label{Li}
\log g(n)=\sqrt{\Li^{-1}(n)} +\co(\sqrt n \exp(-a\sqrt{\log n}))
\end{equation}
for some $a>0$; $\Li^{-1}$ denotes the inverse function of the 
integral logarithm.

The survey paper \cite{Mil} of W. Miller is a nice introduction to $g(n)$;
it contains elegant and simple proofs of \eqref{g}, \eqref{ErTu} and 
\eqref{gasy}.

J.-P. Massias proved in \cite{Mas} that for $n \ge 1$
\begin{equation}\label{gMas}
\log g(n) \le
\frac{\log g(1319366)}{\sqrt{1319366\,\log(1319366)}} \sqrt{n\log n}
\approx 1.05313\sqrt{n\log n}.
\end{equation}
In \cite{MNR2} more accurate 
effective results are given, including
\begin{equation}\label{gMasmin}
\log g(n) \ge  \sqrt{n\log n},\qquad n\ge 906
\end{equation}
and
\begin{equation}\label{gMasmaj}
\log g(n) \le  \sqrt{n\log n}\left( 1+\frac{\log\log n-0.975}{2\log n}\right),
\qquad n\ge 4.
\end{equation}
Let $P^+(g(n))$ denote the greatest prime factor of $g(n)$. In \cite{Gra}, 
J. Grantham proved
\begin{equation}\label{Pgn}
P^+(g(n)) \le 1.328 \sqrt{n\log n}, \qquad n\ge 5.
\end{equation}
Some other functions similar to $g(n)$ were studied in \cite{Ger}, 
\cite{Lev}, \cite{N73}, \cite{Sza} and \cite{Vit}. 

%--------------------------------------------------------------------------
\subsection{Computing Landau's function}
%--------------------------------------------------------------------------

A table of Landau's function up to $300$ is given at the end of \cite{N2}. 
It has been computed with the algorithm described and used in \cite{N4}  
to compute $g(n)$ up to $8000$. By using similar algorithms, a table 
up to $32000$ is given in \cite{Mor}, and a table up to $500000$ is 
mentioned in \cite{Gra}. The algorithm given in \cite{N4} will be referred 
in this paper as the basic algorithm. We shall recall it in Section 
\ref{basic}. 
It can be used to compute $g(n)$ for $n$ up to, say, one million, eventually a 
little more. It cannot compute $g(n)$ without calculating simultaneously
$g(n')$ for $1\le n' \le n$.

If we look at a table of $g(n)$ for $31000\le n\le 31999$ (such a table 
can  be easily built by using the Maple procedure given in Section 
\ref{basic}), we observe 
three parts among the prime divisors of $g(n)$. More precisely, let us set
\[ 
g(n)=\prod_p p^{\al_p},\; g^{(1)}(n)=\prod_{p\le 17} p^{\al_p},\;
g^{(2)}(n)=\prod_{19\le p\le 509} p^{\al_p},\;
g^{(3)}(n)=\prod_{p > 509} p^{\al_p};
\]
the middle part $g^{(2)}(n)$ is constant (and equal to 
$\prod_{19\le p\le 509} p$) for all $n$ between $31000$ 
and $31999$, while the first part $g^{(1)}(n)$ takes only 18 values, 
and the third part $g^{(3)}(n)$ takes 92 values.

So, if $n'$ is in the neighbourhood of $n$, $g(n')/g(n)$ is a 
fraction which is the product of a {\it prefix} (made of small primes)
and a {\it suffix} (made of large primes).

The aim of this article is to make precise this remark to get an algorithm 
able to compute $g(n)$ for some fixed $n$ up to $10^{15}$.

%--------------------------------------------------------------------------
\subsection{The new algorithm}\label{stepalgo}
%--------------------------------------------------------------------------
Let $\tau(n)=\sum_{d\dv n}1$ be the divisor function. To study 
{\it highly composite numbers} (that is the $n$'s such 
that $m < n$ implies $\tau(m) < \tau(n)$), S. Ramanujan 
(cf. \cite{Ram,Ram2,N39}) 
has introduced the {\it superior highly composite numbers} which maximize
$\tau(n)/n^\vep$ for some $\vep > 0$. This definition can be 
extended to function $\ell$: $N$ is said to be $\ell$-superchampion if 
it minimizes $\ell(N)-\rho\log(N)$ for some $\rho > 0$. These numbers 
will be discussed in Section~\ref{superch}: they are easy to compute and have 
the property that, if $n=\ell(N)$, then $g(n)=N$.

If $N$ minimizes $\ell(N)-\rho\log(N)$, we call {\it benefit} of an 
integer $M$ the non-negative quantity $\ben(M)=\ell(M)-\ell(N)-
\rho\log(M/N)$. If $n$ is not too far from $\ell(N)$, 
a relatively small bound can be obtained for
$\ben g(n)$, and this allows computing it. This notion of benefit 
will be discussed in Section~\ref{benefit}.
\medskip

To compute $g(n)$, the main steps of our algorithm are
\begin{enumerate}
\item Determine the two consecutive $\ell$-superchampion numbers $N$ and $N'$
such that $\ell(N) \le n < \ell(N')$ and their common parameter $\rho$ 
(cf. Section \ref{compsup}).

\item For a guessed value $B'$, determine a set $\cd(B')$ of plain
prefixes whose benefit is smaller than $B'$ (cf. Section \ref{prefi1} and 
Section \ref{compplain}).

\item Use the set $\cd(B')$ to compute an upper bound $B$ such that
$\ben g(n)\le \ben g(n)+n-\ell(g(n))\le B$ (cf. Section \ref{upbb}); 
note that, from \eqref{lgn}, $\ell(g(n))\le n$ holds.

\item Determine $\cd(B)$, a set containing the plain prefix of $g(n)$.
If $B < B'$, to get $\cd(B)$, we just 
have to remove from $\cd(B')$ the elements whose benefit is bigger than $B$.
If $B > B'$, we start again the algorithm described in Section 
\ref{compplain} to get $\cd(B')$ with a new value of $B'$ greater
than $B$.

\item Compute a set containing the normalized prefix of $g(n)$ 
(cf. Sections \ref{compnorm}, \ref{heart} and \ref{fight}).

\item Determine the suffix of $g(n)$ by using the function $G(p_k,m)$ 
introduced in Section \ref{intrG} and discussed in Sections \ref{grandpe} 
and \ref{Glarge}.
\end{enumerate}
In the sequel of our article, `` step '' will refer to one of the
above six steps, and `` the   algorithm '' will refer to the algorithm
sketched in Section \ref{stepalgo}.

On the web site  of the second author, there is a {\sc Maple} code of this 
algorithm where each instruction is explained according with the 
notation of this article.

If we want to calculate $g(n)$ for consecutive values $n=n_1, n=n_1+1,\ldots,
n=n_2$, most of the operations of the algorithm are similar and can be put in 
common; however, due to some technical questions, it is more difficult to
treat this problem, and here, we shall restrict ourselves to the computation of
$g(n)$ for one value of $n$.

To compute the first 5000 highly composite numbers, G. Robin (cf. \cite{Rob})
already used a notion of benefit similar to that introduced in this article.

%--------------------------------------------------------------------------
\subsection{The function $G(p_k,m)$}\label{intrG}
%--------------------------------------------------------------------------

In step 6, the computation of the suffix of $g(n)$ 
leads  to the function $G(p_k,\!m)$, defined by

\begin{definition}
Let $p_k$ be the $k$-th prime, for some $k\ge 3$ and $m$ an
integer satisfying $0 \le m \le p_{k+1}-3$.
We define
\begin{equation}\label{G1}
G(p_k,m)=\max \frac{Q_1Q_2\ldots Q_s}{q_1q_2\ldots q_s}
\end{equation}
where the maximum is taken over the primes 
$Q_1,Q_2,\ldots,Q_s,q_1,q_2,\ldots,q_s$ ($s\ge 0$) satisfying 
\begin{equation}\label{qsQ}
3\le q_s < q_{s-1} < \ldots < q_1 \le p_k < p_{k+1} \le Q_1 < Q_2 < 
\ldots < Q_s
\end{equation}
and
\begin{equation}\label{Qimqi}
\sum_{i=1}^s (Q_i-q_i)\le m.
\end{equation}
\end{definition} 

This function $G(p_k,m)$ is interesting in itself.
It satisfies
\begin{equation}\label{lGm}
\ell(G(p_k,m))\le m.
\end{equation}
We study it in Section
\ref{grandpe}, where a combinatorial algorithm is given to compute its value
when $m$ is not too large. For $m$ large, a better 
algorithm is given in Section \ref{Glarge}.

Let us denote by $\mu_1(n) < \mu_2(n) < \ldots$ the increasing 
sequence of the primes
which do not divide $g(n)$, and by $P(n)$ the largest prime factor
of $g(n)$. It is shown in \cite{N1} that $\lim_{n\to\iy}P(n)/\mu_1(n)=1$.
We may guess from Proposition \ref{Gnew} that $\mu_1(n)$ can be 
much smaller than $P(n)$ while $\mu_2(n)$ is closer to $P(n)$. It seems 
difficult to prove any result in this direction.

%--------------------------------------------------------------------------
\subsection{The running time}
%--------------------------------------------------------------------------

Though we have the feeling that the algorithm presented in this paper 
(and implemented in {\sc Maple}) yields the value of $g(n)$ for all $n$'s up 
to $10^{15}$ (and eventually for greater $n$'s) in a reasonable time, 
it is not proved to do so.

Indeed, we do not know how to get an effective upper bound for the 
{\it benefit} of $g(n)$ (see sections \ref{benefit}, \ref{upbb} and 
\ref{openben}) and in the second and third
steps, what we do is just, for a given $n$, to provide such an upper bound
$B=B(n)$ by an experimental way.

In the fourth step, the algorithm determines a set $\cd(B)$ of plain
prefixes (cf. sections \ref{compplain} and \ref{upbb}). 
It turns out that the number $\nu(n)$ of these prefixes is 
rather small and experimentally satisfies $\nu(n)=O(n^{0.3})$ 
(cf. \eqref{n0.3});
but we do not know how to prove such a result, and it might
exist some values of $n$ for which $\nu(n)$ is much larger.

Let us now analyze each of the six steps described in Section
\ref{stepalgo}.

The first step determines the greatest superchampion number $N$ such that
$\ell(N)\le n$.
Let $S(x)=\sum_{p \le x} p$ be the sum of the primes up to $x$. The main 
part of this step is to compute $S(x)$ for $x$ close to $\sqrt{n\log n}$.
In our {\sc Maple} program, by Eratosthenes' sieve, we have precomputed 
a function close to $S(x)$, the details are given in Section \ref{compsup}. 
However, a faster way exists to evaluate $S(x)$. By extending 
Meissel's technique to compute $\pi(x)=\sum_{p \le x} 1$, 
(cf. \cite{DelRiv}), M. Del\'eglise is able to compute
$\sum_{p \le x} f(p)$ where $f$ is a multiplicative function. E. Bach 
(cf. \cite{Bach,BacSor}) has considered a wider class of functions for 
which this 
method also works. By his algorithm, M. Del\'eglise has computed $S(10^{18})$, 
and $S(x)$ costs $O(x^{2/3}/\log^2 x)$. We hope to implement soon this 
new evaluation of $S(x)$ in our first step.

The second and the fourth steps compute respectively $\cd(B')$ and
$\cd(B)$.  If $B'$ is ``well'' chosen, we may hope that
Card$(\cd(B'))$ is not much larger than $\nu(n)={\rm Card}(\cd(B))$.
The running time of the computation of $\cd(B')$ as explained in
Section \ref{compplain} could be larger than $\nu(n)$. For $n\approx
10^{20}$, most of the time of the computation of $g(n)$ is spent in
the second and fourth steps. But any precise estimation of these steps
seems unaccessible.

The running time of the third step is $O({\rm Card}(\cd(B')))$, 
and we may hope that it is $O(\nu(n))$.

In practice, the fifth step (finding the possible normalized prefixes) is 
fast. For every plain prefix $\pih$, Inequations \eqref{Som1} have at most
one solution, and the cost of this step is $O(\nu(n))$.

The sixth and last step also is fast. Under the strong assumption that
$\de_1(p)$ is polynomial in $\log p$ (see \eqref{m1maj}), 
for any $m$, the computation of
$G(p,m)$ (where $p$ is a prime satisfying $p\approx \sqrt{n\log n}$)
is polynomial in $\log n$, and the number of
normalized prefixes surviving the fight (cf. Section \ref{fight})
seems to be bounded (we have no examples of more than three of them), so that (see Section
\ref{heart}) 
this step might be polynomial in $\log n$.

%--------------------------------------------------------------------------
\subsection{Plan of the paper}
%--------------------------------------------------------------------------

In Section \ref{lemmas}, some mathematical lemmas are given.
The various steps of the algorithm presented in Section \ref{stepalgo} 
are explained in sections  \ref{superch}-\ref{Glarge}; 
Section \ref{perfnewalgo}  
presents some results while Section \ref{openpb} asks 
five open problems.

%--------------------------------------------------------------------------
\subsection{Notation}
%--------------------------------------------------------------------------

We denote by $\cp=\{2,3,5,7,\ldots\}$ the set of primes, by 
$p\in \cp$ a generic prime, by $p_i$ the $i$-th prime
and by $v_p(N)$ the $p$-adic valuation of $N$, that is the greatest integer
$\al$ such that $p^\al$ divides $N$. $Q_i$ and $q_i$ also denote primes,
except in Lemma \ref{lemme1} which is stated in a 
more general form, but 
which is used with $Q_i$ and $q_i$ primes. The integral part of 
a real number $t$ is denoted by $\fl{t}$. The additive function
$\ell$ can be easily extended to a rational number by setting
$\ell(A/B)=\ell(A)-\ell(B)$ (with $A$ and $B$ coprime).

%--------------------------------------------------------------------------
\section{The basic algorithm}\label{basic}
%--------------------------------------------------------------------------

%--------------------------------------------------------------------------
\subsection{The first version}
%--------------------------------------------------------------------------

For $j\ge 0$, let us denote by $\cs_j$ the set of numbers having only
$p_1,p_2,\ldots,p_j$ as prime divisors
\begin{equation}\label{Sj}
\cs_j=\{M; p\dv M \; \LR \; p\le p_j\}.
\end{equation}
We have $\cs_0=\{1\}$, $\cs_1=\{1,2,4,8,16,\ldots\}$. The algorithm described
in \cite{N4} computes the functions
\begin{equation}\label{gj}
g_j(n)=\max_{M\in\cs_j,\ \ell(M)\le n} M
\end{equation}
which obviously satisfy the induction relation
\begin{equation}\label{gjrec}
g_j(n)=\max\left[g_{j-1}(n),p_jg_{j-1}(n-p_j),\ldots,
p_j^kg_{j-1}(n-p_j^k)\right]
\end{equation}
where $k$ is the largest integer such that $p_j^k\le n$, and 
$g_0(n)=1$ for all $n\ge 0$. Using the upper bound \eqref{Pgn}, we
write the following {\sc Maple} procedure:

\renewcommand{\algorithmicendfor}{\textbf{endo}}

\begin{algorithm}\label{basicalg}
\caption{The basic algorithm: this {\sc Maple} procedure computes
$g(n)$ for $0 \le n \le N$ and stores the results in table $g$.}
\begin{algorithmic}
\STATE{gden:=  proc(N)}
local $n,g,pmax,p,k,a$
\FOR{$n$ from $0$ to $N$}
\STATE{$g[n]:= 1$}
\ENDFOR;
\STATE{$pmax:= floor(1.328 \star eval(sqrt(N \star \log N)));$}
\STATE{$p:= 2;$}
\WHILE{$p \le pmax$}
\FOR{$n$ {\bf from} $N$ {\bf to} $p$ {\bf by} $-1$}
\FOR{$k$ {\bf from} $1$ {\bf while} $p^k \le n$}
\STATE{$a:=p^k\star g[n-p^k];$}
\IF{$g[n] < a$}
\STATE{$g[n]:=a$}
\ENDIF
\ENDFOR
\ENDFOR;
\STATE{p:=nextprime(p)}
\ENDWHILE;\\
\textbf{end;}
\end{algorithmic}
\end{algorithm}

The running time of this procedure 
is 13 hours for $N=10^6$ on a 3 Ghz Pentium 4 with a storage of 337 Mo.
To compute $g(n)$, $1 \le n\le N$, the theoretical running time is 
$\co\left(N^{3/2}/\sqrt{\log N}\right)$
and the needed memory is $\co(N)$ integers of size $\exp(O(\sqrt{N\log N}))$. 

%--------------------------------------------------------------------------
\subsection{The merging and pruning algorithm}\label{merpru}
%--------------------------------------------------------------------------

\ni
The above algorithm takes a very long time to compute $g_j(n)$ 
when $j$ is small. It is better to represent $(g_j(n))_{n\ge 1}$ by a list 
$L_j=[[M_1,l_1],\ldots,[M_i,l_i],\ldots]$ 
(where $l_i=\ell(M_i)$) ordered so that $M_{i+1}>M_i$ and $l_{i+1} > l_i$. If 
$l_i \le n< l_{i+1}$, then $g_j(n)=M_i$. So, $L_0=[[1,0]]$ and 
$L_1=[[1,0],[2,2],[4,4],[8,8],\ldots]$. 

To calculate $L_{j+1}$ from $L_{j}$ we construct the list of all elements
$[M_i p_{j+1}^a,l_i+\ell(p_{j+1}^a)]$ for all elements $[M_i,l_i]\in L_j$
and $a\ge 0$ such that $l_i+\ell(p_{j+1}^a) \le N$. We sort 
this new list with respect to the first term of the elements 
(merge sort is here specially recommended) to get a list $\La=[[K_1,\la_1],
[K_2,\la_2],\ldots]$ with $K_1 < K_2 < \ldots$ Now, to take \eqref{gjrec}
into account, we have to prune
the list $\La$: if $K_r < K_s$ and 
$\la_r \ge \la_s$, we take off the element $[K_r,\la_r]$ from the list $\La$.
The list $L_{j+1}$ will be the pruned list of $\La$.

%--------------------------------------------------------------------------
\section{Two lemmas}\label{lemmas}
%--------------------------------------------------------------------------

\begin{lem}\label{lemme1}
Let $s$ be a non-negative integer, and $t_1,q_1,q_2,\ldots,
q_s,Q_1,Q_2,\ldots,$ $Q_s$ be real numbers satisfying
\begin{equation}\label{hyplem1}
0<t_1\le q_s<q_{s-1}<\ldots<q_1<Q_1<Q_2<\ldots <Q_s.
\end{equation}
If we set \dsm{S=\sum_{i=1}^s Q_i-q_i},
then the following inequality holds:
\[
\frac{Q_1 Q_2\ldots Q_s}{q_1 q_2\ldots q_s} \le 
\exp \left(\frac{S}{t_1}\right). \leqno{1.}
\]
Moreover, if $s\ge 1$ and $S < Q_1$, we have
\[
\frac{Q_1 Q_2\ldots Q_s}{q_1 q_2\ldots q_s} \le \frac{Q_s}{Q_s-S} 
< \frac{Q_{s-1}}{Q_{s-1}-S} < \ldots < \frac{Q_1}{Q_1-S} \leqno{2.}
\]   
with the first inequality in 2. strict when $s\ge 2$.   
\end{lem}

\begin{proof}
Lemma \ref{lemme1} is a slight improvement of 
Lemma 3 of \cite{N2} where, in 2., only the upper bound $Q_1/(Q_1-S)$ 
was given.
Point 1. is easy by applying $1+u\le \exp u$ to $u=Q_i/q_i-1$. 
Let us prove  2. by induction. For $s=1$, 2. is an equality. 
Let us assume that $s\ge 2$. Setting $S'=\sum_{i=2}^s Q_i-q_i=S-(Q_1-q_1)$, 
we have $S'<S<Q_1<Q_s$ and by induction hypothesis, we get
\begin{equation}\label{p1lem1}
\frac{Q_1 Q_2\ldots Q_s}{q_1 q_2\ldots q_s} = \frac{Q_1}{q_1}\;
\frac{Q_2\ldots Q_s}{q_2\ldots q_s}\le \frac{Q_1}{q_1}\;\frac{Q_s}{Q_s-S'}\cdot
\end{equation}
\ni
We shall use the following principle:
\begin{principe} 
If $x$ and $y$ add to a constant, the product $xy$ 
decreases when $|y-x|$ increases.    
\end{principe}

We have $Q_s-S'\le Q_s-(Q_s-q_s)=q_s < q_1$, and
using Principle 1, we get by increasing $q_1$ to $Q_1$ and
decreasing $Q_s-S'$ to $Q_s -S$
\[
q_1(Q_s-S') > Q_1(Q_s-S)
\]
which, from \eqref{p1lem1}, proves 2..
\end{proof}

\begin{lem}\label{lemx2x1}
Let $x > 4$ and $y=y(x)$ be 
defined by
\dsm{
\frac{y^2-y}{\log y}=\frac{x}{\log x}\cdot
}
The function $y$ is an increasing function satisfying $y(x) > 2$ and 
\begin{enumerate}
\item
\dsm{
y(x)=\sqrt{\frac x2}\left(1-\frac{\log 2}{2\log x}-
\frac{(4+\log 2)\log 2}{8\log^2 x}
+\co\left(\frac{1}{\log^3 x}\right) \right), \quad x\to\iy
}
\item
$y(x) < \sqrt x$  for $x > 4$.   
\item
\dsm{y(x) \le \sqrt{\frac x2}} for $x\ge 80.$
\end{enumerate}
\end{lem}

\begin{proof}
1. and 3. are proved in \cite{MNR1}, p. 227. 
Since $t\mapsto (t^2-t)/\log t$ is increasing for $t>1$, in order
to show 2., one should prove 
\dsm{\frac{x -\sqrt{x}}{\frac 12 \log x}  > \frac{x}{\log x}} which holds 
for $x > 4$. 
\end{proof}

%The function $y(x)$ defined in Lemma \ref{lemx2x1} below plays an 
%important role in the study of $g(n)$. We recall its behaviour when $x\to\iy$.

%--------------------------------------------------------------------------
%\section{${\mathbb \ell}$-superchampion numbers\label{superch}}
\section{The superchampion numbers}\label{superch}
%--------------------------------------------------------------------------
%--------------------------------------------------------------------------
%\subsection{Definition and some properties}
%--------------------------------------------------------------------------

%--------------------------------------------------------------------------
%\subsection{Minimizing an additive function}
%--------------------------------------------------------------------------

\begin{definition}
An integer $N$ is said $\ell$-superchampion (or more 
simply superchampion) if there exists $\rho>0$ such that, for all $M\ge 1$
\begin{equation}\label{super}
\ell(M)-\rho \log M \ge \ell(N) -\rho \log N.
\end{equation} 
When this is the case, we say that $N$ is a  $\ell$-superchampion
\emph{associated} to $\rho$.
\end{definition}

Geometrically, if we represent $\log M$ in abscissa and $\ell(M)$ in ordinate,
the straight line of slope $\rho$ going through the point $(\log M,\ell(M))$ 
has an intersep equal to $\ell(M)-\rho\log(M)$ and so, the superchampion 
numbers are the vertices of the convex envelop of all these 
points (see {\sc Fig}. \ref{figconvexe}).
\begin{figure}
\hspace{1.5cm}
\includegraphics[height=90mm]{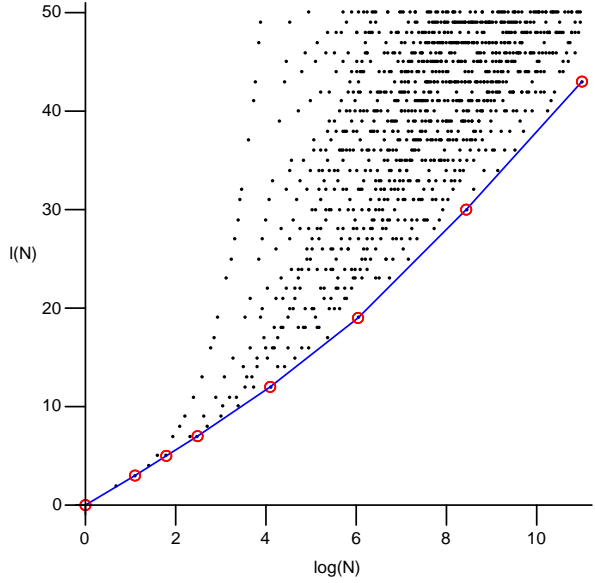}
\caption{The points $(\log(N),\ell(N))$, with $\ell(N) \le 50$, for 
$1 \le N \le 60060$.}
\label{figconvexe}
\end{figure}

Similar numbers, the so-called {\it superior highly composite numbers} were 
first introduced by S. Ramanujan (cf. \cite{Ram}). The $\ell$-superchampion
numbers were already used in \cite{N1,N2,Mas,MNR1,MNR2,N60,N73}.
The first ones are (with, in the third column, the corresponding values
of $\rho$) shown in  {\sc Fig}. \ref{firstchampions}.

\begin{lem}
If $N$ is an $\ell$-superchampion, the following property holds:
\begin{equation}\label{NgN}
N=g(\ell(N)).
\end{equation}
\end{lem}

\begin{proof}
Indeed, let $N$ be any positive number and $n=\ell(N)$; 
it follows from \eqref{pcarg} that $N\le g(n)=g(\ell(N))$. 
If moreover $N$ is a 
$\ell$-superchampion, then, for all $M$ such that 
$\ell(M)\le n=\ell(N)$, from \eqref{super}, we have
$\rho\log M \le \rho \log N +\ell(M)-\ell(N)
\le \rho \log N$ which implies $M\le N$, and thus, from \eqref{g}, 
\eqref{NgN} holds.
\end{proof}

\begin{definition}
%%Let be $\rho > \dfrac{2}{\log 2}=\dfrac{4}{\log 4} > \exp(1)$.
\mbox{}
\begin{enumerate}
\item
For each prime $p\in \cp$, let us define the sets
\begin{equation}\label{Ep}
\ce_p' = \set{\frac{p}{\log p}},
\quad
\ce_p\sg =\left\{\frac{p^2-p}{\log p}\,,\ldots,
\frac{p^{i+1}-p^i}{\log p}\,,\ldots \right\},
\quad
\ce_p = \ce_p' \cup \ce_p\sg.
\end{equation}
\item
And we define
\begin{equation}\label{Edef}
\ce'= \bigcup_{p \in \Primes} \ce_p',
\quad
\ce\sg =\bigcup_{p \in \Primes} \ce_p\sg
\quad\text{and}\quad
\ce = \ce' \cup \ce\sg.
\end{equation}
\end{enumerate}
\end{definition}

\begin{rem}
Note that all the elements of $\ce_p$ are distinct at the exception, for $p=2$,
of $\dfrac{2}{\log 2}=\dfrac{2^2-2}{\log 2}$ and that, for 
$p \ne q$, $\ce_p\cap \ce_q=\emptyset$ holds. 
\end{rem}

\begin{lem}\label{lemNro1}
Let $\rho$ a real number. %\dsm{\rho \ne \frac{2}{\log 2} \approx 2.88}.
\begin{enumerate}
\item
If $\rho \in \ce_p$, $\rho \ne \frac{2}{\log 2}$,
there exist exactly $2$ superchampion
numbers associated to $\rho$. Let be $N_\rho$ the smaller
one and $N_\rho^+$ the bigger one. Then
$N_\rho^+ = p N_\rho$ and
\begin{equation}\label{Nrho}
N_\rho=\prod_{p/\log p\, < \rho} p^{\al_p} 
\qtx{with}
\al_p\! = \!\left\{
\begin{array}{ll}
1 & \qtx{if}
\dfrac{p}{\log p} < \rho \le \dfrac{p^2-p}{\log p}\\[3ex]
i & \qtx{if}
\dfrac{p^i-p^{i-1}}{\log p} < \rho \le \dfrac{p^{i+1}-p^i}{\log p}
\end{array}
\right.
\end{equation}
\begin{equation}\label{Nrhop}
N_\rho^+=\prod_{p/\log p\, \le\rho} p^{\al_p^+}
\qtx{with}
\al_p^+ \! = \!\left\{
\begin{array}{ll}
1&  \qtx{if}
\dfrac{p}{\log p} \le \rho < \dfrac{p^2-p}{\log p}\\[3ex]
i &\qtx{if}
\dfrac{p^i-p^{i-1}}{\log p} \le \rho < \dfrac{p^{i+1}-p^i}{\log p}
\end{array}
\right.
\end{equation}
\item
If \dsm{\rho = \frac{2}{\log 2}=\frac{2^2-2}{\log 2}\in\ce},
there exist $3$ superchampion numbers associated to $\rho$: $N_\rho$
defined by \eqref{Nrho} is equal to $3$, $N_\rho^+$ defined by \eqref{Nrhop} 
is equal to $12$ and the third one is $6$. 
\item
If $\rho \not \in \ce$, there exists a unique superchampion
number $N_\rho=N_\rho^+$ associated to $\rho$. Its value is given by 
both formulas \eqref{Nrho} and \eqref{Nrhop}. Let $\rho'$ and $\rho''$ be
the two consecutive elements of $\ce$ such that $\rho'< \rho < \rho''$. 
Then we have $N_\rho=N_{\rho''}=N_{\rho'}^+$.
\item
Let us consider the sequence $\rho^{(i)}$ defined by $\rho^{(0)}=-\iy$,
$\rho^{(1)}=3/\log 3$, $\rho^{(2)}=2/\log 2$,
$\rho^{(3)}=(2^2-2^1)/\log 2=\rho^{(2)}$, $\rho^{(4)}=5/\log 5$ and such that
$\set{\rho^{(i)}, i\ge 1}=\ce$ and $\rho^{(i)} > \rho^{(i-1)}$ for $i\ge 4$.
If $N^{(0)}=1$, $N^{(1)}=3$, $N^{(2)}=6$, $N^{(3)}=12$, $N^{(4)}=60$, etc...
is the increasing sequence of all superchampion numbers, it satisfies:
\begin{enumerate}
\item
 For $i\ge 0$, $N^{(i)}$ divides $N^{(i+1)}$ and the quotient 
$N^{(i+1)}/N^{(i)}$ is a prime number. The number of prime factors of
$N^{(i)}$, counting them with multiplicity, is equal to $i$.
\item
For $i\ne 2$, we have $N^{(i)}=N_{\rho^{(i)}}^+=N_{\rho^{(i+1)}}$ where
$N_{\rho^{(i)}}^+$ and $N_{\rho^{(i+1)}}$ are defined respectively in 
\eqref{Nrho} and \eqref{Nrhop}. 
\item
For all $i\ge 0$, $N^{(i)}$ is associated to $\rho$ if and only if
$\rho^{(i)}\le \rho \le \rho^{(i+1)}$.
\item
If $i\ne 1$ (i.e., $N^{(i)}\ne 3$), then $v_p(N^{(i)})$ is a non-increasing 
function of the prime $p$. 
\end{enumerate}
 \end{enumerate}
\end{lem}

\begin{proof}
We are looking for an 
$N = \prod p^{\al_p}$ 
which minimizes 
$F(N) =  \ell(N) - \rho \log N$.

 An arithmetic function $h$ is said \emph{additive} if
  $h(M_1M_2)=h(M_1)+h(M_2)$ when $M_1$ and $M_2$ are coprime. The functions
  $\log$ and $\ell$ are additive. Thus $F$ is additive, and
  to minimize $F(N)= \sum_{p \dv N} F(p^{v_p(N)})$ we have to minimize 
  $F(p^\al)$ on $\al$ for each $p\in\cp$.  We have $F(1) = 0$ and 
  for $p$ prime and $i \ge 1$, $F(p^i)=p^i-\rho\, i\log p$. The difference
\begin{equation}\label{fpi}
F(p^{i+1}) - F(p^i) = 
\begin{cases}
p - \rho\log p &\text{ if } i = 0\\[2mm]
p^{i}(p-1) -\rho\log p &\text{ if } i > 0
\end{cases}
\end{equation}
is a non decreasing function of $i$ that tends to $+\infty$ with $i$.
Thus if $F(p) = F(p)-F(0) = p-\rho\log p > 0$, 
the smallest value of $F(p^\al)$ is $0$
obtained for $\al=0$.  If $F(p) \le 0$ let $i$ be the largest positive
integer such that $F(p^i) - F(p^{i-1}) \le 0$. Then the smallest value of
$F(p^\al)$ is obtained on the set $\setof{j \le i}{F(p^j) = F(p^i)}$
and the number of choices for $\al_p$ is the cardinal of this set.

This proves that we have more than one choice for the exponent $\al_p$
if and only if there exists $i \ge 0$ such that $F(p^i) = F(p^{i+1})$.
Due to \eqref{fpi} this is the case if and only if  $\rho \in \ce_p$.
Moreover, the sets $\ce_p$ being disjoint, there exists at most one $p$
for which there are more than one choice for $\al_p$. 

If $p \ge 3$ we have 
$p < (p^2-p) < (p^3-p^2) < \cdots$
and there is at most one $i$ such that $F(p^{i+1})-F(p^i)=0$,
so there are at most two choices for $\al_p$.

For $p = 2$ we have $2 = 2^2-2 < 2^3-2^2 < \cdots$
and for $\rho = 2/\log 2$  we have $F(1) = F(2) = F(2^2)$,
so we can choose for $\al_2$ every one of the three values $0,1,2$.
With this value of $\rho$ we have 
$F(3) = 3-(2/\log2)\log 3 < 0$ and $F(p) > 0$ for $p \ge 5$.
Thus there are 3 superchampion numbers associated to $\rho = 2/\log 2$ 
which are $3,6,12$.
This proves 1., 2., 3. and 4.; for more details, see \cite{N2}.

\end{proof}

\begin{figure}
\[
\begin{array}{|r|r|lccl|l|}
\hline
N  &\ell(N)&       &    &     &             &\\[0.5mm]
\hline
1  & 0 & -\iy      &  < &\rho &\le 3/\log 3 &\approx 2.73\\[0.5mm]
3  & 3 & 3/\log 3  &\le &\rho &\le 2/\log 2 &\approx 2.89\\[0.5mm]
6  & 5 &           &    &\rho &=(2^2-2^1)/\log 2&\approx 2.89\\[0.5mm]
12 & 7 & 2/\log 2  &\le &\rho &\le 5/\log 5&\approx3.11\\[0.5mm]
60 &12 & 5/\log 5  &\le &\rho &\le 7/\log 7&\approx3.60\\[0.5mm]
420&19 & 7/\log 7  &\le &\rho &\le 11/\log 11&\approx4.59\\[0.5mm]
4620&30 &11/\log 11  &\le &\rho &\le 13/\log 13&\approx5.07\\[0.5mm]
60060&43 & 13/\log 13  &\le &\rho &\le (3^2-3^1)/\log 3&\approx5.46\\[0.5mm]

\hline
\end{array}
\]
\caption{The first $\ell$-superchampion numbers.}
\label{firstchampions}
\end{figure}

\begin{lem}\label{defxj}
Let $\rho$ satisfy  $\rho \ge 5/\log 5 \approx 3.11$. There exists a unique
decreasing sequence $(x_j)=(x_j(\rho))$ such that 
$x_1 \ge \exp(1)$ and, for all $j \ge 2$,
$x_j$ satisfies $x_j > 1$ and
\begin{equation}\label{defxi}
%\forall j \ge 2
\quad
\frac{x_j^j-x_j^{j-1}}{\log x_j}
=\frac{x_1\vphantom{x_j^j98834}}{\ds\log x_1}=\rho.
\end{equation}
We have also
\begin{equation}\label{x14}
x_1 \ge 5
\qtx{ and }
x_2 > 2.
\end{equation}
\end{lem}

\begin{proof}
The uniqueness of $x_1$ results from $\rho > \exp(1)$ and the fact that
$t\mapsto t/\log t$ is an increasing bijection of $[\exp(1),+\infty[$.
The uniqueness of $x_j$ for $j \ge 2$ comes from the fact
that $t \mapsto (t^j-t^{j-1})/\log t=t^{j-1}(t-1)/\log t$ is 
an increasing bijection of  $]1,+\infty[$.
The inequality $x_j > x_{j+1}$ for $j \ge 2$ comes from the
increase of \dsm{j \mapsto (t^j-t^{j-1})/\log t} for
each $t > 1$.

Let us prove that $x_1 > x_2$. The definition \eqref{defxi} 
of $x_2$ implies 
\[
\frac{x_2^2-x_2}{\log x_2} = \rho > \frac{2}{\log 2} =
  \frac{2^2-2}{\log 2}\approx 2.89\;.
\]
With the increase of  $t \mapsto (t^2-t)/\log t$ 
this proves $x_2 > 2$. Thus $x_2^2-x_2 > x_2$, and
therefore
\[
\frac{x_2}{\log x_2} < \frac{x_2^2-x_2}{\log x_2} = \rho =
\frac{x_1}{\log x_1}
\]
which, with the increase of \dsm{t \mapsto t/\log t} on
$[\exp(1),+\infty[$ yields $x_2 > x_1$ and the decrease of $(x_n)$.
Finally $x_1/\log x_1 = \rho \ge 5/\log 5$ gives $x_1 \ge 5$.
\end{proof}

\begin{prop}
Let $\rho$ be a real number satisfying $\rho \ge 5/\log 5$, 
$N_\rho$ the smallest superchampion number
associated to $\rho$ and $N_\rho^+$ the largest superchampion number
associated to $\rho$ (cf. Lemma \ref{lemNro1}).
Then, with $x_j$ as introduced in Lemma \ref{defxj}, we have
\begin{equation}\label{Nxj}
N_{\rho}=\prod_{j\ge 1}\quad \prod_{x_{j+1}\le p < x_j} p^j 
\quad \text{ and } \quad
N_\rho^+=\prod_{j\ge 1} \quad\prod_{x_{j+1} <  p \le x_j} p^j. 
\end{equation}
\end{prop}

\begin{proof}
Due to \eqref{Nrho}, $\al_p = 1$ holds if and only we have
\begin{equation}\label{exp1}
\dfrac{p}{\log p} < \rho \le \dfrac{p^{2}-p}{\log p},
\end{equation}
and by the definition \eqref{defxi} of $x_1$ and $x_2$, this is 
equivalent to
\[
\dfrac{p}{\log p} < \frac{x_1}{\log x_1}
\qtx{and}
\dfrac{x_{2}^{2}-x_2}{\log x_{2}} 
\le \dfrac{p^{2}-p}{\log p}\cdot
\]
By the increase of $t \mapsto t/\log t$ on
$[\exp(1),+\infty[$ and $t \mapsto (t^2-t)/\log t$ on
$[1,+\infty[$,
this proves that for $p \ge \exp(1)$, $\al_p = 1$ holds if and only
if $x_2 \le p  < x_1$. 
It remains to prove that, when $p=2$, this equivalence is still
true. In this case, $2/\log 2 = (4-2)/\log 2$, and \eqref{exp1}
is never satisfied. By \eqref{x14} we have $x_2 > 2$,
and $x_2 \le 2 < x_1$ is false.
Thus, for every prime $p$, we have 
$\al_p = 1$ if and only if $x_{2} \le p < x_1$.

For $i \ge 2$, $\al_p = i$ if and only if
\dsm{
\dfrac{p^i-p^{i-1}}{\log p} < \rho \le \dfrac{p^{i+1}-p^i}{\log p},
}
and, by the definition \eqref{defxi} of $x_i$ and $x_{i+1}$
this is equivalent to
\[
\dfrac{p^i-p^{i-1}}{\log p} < 
\dfrac{x_i^i-x_{i}^{i-1}}{\log x_i} 
\qtx{and}
\dfrac{x_{i+1}^{i+1}-x_{i+1}^{i}}{\log x_{i+1}} 
\le \dfrac{p^{i+1}-p^i}{\log p}
\]
or $x_{i+1} \le p < x_i$. 
This proves the first equality \eqref{Nxj}. The second one can be proved 
by the same way.
\end{proof}

%--------------------------------------------------------------------------
\section{First step of the computation of ${g(n)}$:  
getting $\rho, N, N'$.}\label{compsup}
%such that $\ell(N_\rho) \le  n < \ell(N_\rho^+)$}
%--------------------------------------------------------------------------

%--------------------------------------------------------------------------
\subsection{Fixing our notation}
%-------------------------------------------------------------------------

\begin{figure}
\[
\begin{array}{|r||r|r|r|r|}
\hline
i  & T[i].q     &  T[i].j  & T[i].p    & T[i].\ell \\[0.5mm]
\hline
1  &2     &  2  &   3  &    7 \\
2  &3     &  2  &  13  &   49 \\
3  &2     &  3  &  13  &   53 \\
4  &2     &  4  &  43  &  301 \\
5  &5     &  2  &  47  &  368 \\
6  &3     &  3  &  67  &  626 \\
7  &7     &  2  &  97  & 1160 \\
8  &2     &  5  & 107  & 1487 \\
9  &11    &  2  & 251  & 6307 \\
10 &2     &  6  & 251  & 6339 \\
11 &3     &  4  & 271  & 7453 \\
\hline
\end{array}
\]
\caption{The first elements of table $T$ associated to $E_2$.}
\label{firstchampions2}
\end{figure}

When $\rho = 5/\log 5$ we have $N_\rho = 12$ and $\ell(N_\rho) = 7$
(see {\sc Fig.} 2). 

\begin{definition}\label{fixN}
From now on, $n \ge 7$ will be a fixed integer, and our purpose is
to compute $g(n)$. We will denote by $\rho$ the unique real number $\rho\in\ce$
such that $\rho \ge 5/\log 5$ and 
\begin{equation}\label{NN'}
\ell(N_\rho) \le  n < \ell(N_\rho^+).
\end{equation}
We will also fix the following notation.
\begin{enumerate}
\item
$N = N_\rho$, $\ N' = N_\rho^+$
and \dsm{N = \prod_{p} p^{\al_p}} is the standard factorization of $N$.
\item
We define $x_1 = x_1(\rho)\ge 5$ and $x_2 = x_2(\rho) > 2$ by \eqref{defxi}.
\item
Let $p_k$ be the largest prime factor of $N=N_\rho$.
It follows from \eqref{Nxj} that
\begin{equation}\label{pkx1}
p_k<x_1\le p_{k+1}
\end{equation}
and, actually, $x_1=p_{k+1}$ unless $\rho\in \ce\sg$ (in this case 
$p_k < x_1 <  p_{k+1}$).
\item
Let us define $B_1$  by
\begin{equation}\label{B1}
B_1 = \min\left( x_2^2-2x_2,\frac{x_1}{2}-\sqrt{x_1}\right) > 0.
\end{equation}
\end{enumerate}
\end{definition}

We have
\begin{equation}\label{x1rx2}
2 < x_2 < \sqrt{x_1} < \rho  < x_1.
\end{equation}
Let us prove \eqref{x1rx2}.
Inequalities \eqref{x14} give $2 < x_2$.
With Lemma \ref{lemx2x1}, Point 2., it yields $x_2 < \sqrt{x_1}$.
Since for all $t > 1$, $\sqrt t/\log t > e/2 > 1$ we have
$\sqrt x_1/\log x_1 > 1$ and thus $\rho = x_1/\log x_1 > \sqrt x_1$.

%--------------------------------------------------------------------------
\subsection{The superchampion algorithm}
%-------------------------------------------------------------------------

Given $n$, as already said, the first step in our computation of $g(n)$  is to
calculate $\rho,N,N',x_1,x_2,p_k,B_1$ as introduced in Definition \ref{fixN}.

We begin by precomputing in increasing order the first elements
of $\ce\sg$ and stop when we get the first $r \in \ce\sg$
such that $\ell(N_r^+) > 10^{15}$. We get a set $E_2$
with $1360$ elements,
\[
E_2 = 
\set{\frac{2^2-2}{\log 2},\frac{3^2-3}{\log 3},\frac{2^3-2^2}{\log2}
  ,\cdots}.
\]
 
We construct a table $T$, indexed from $1$ to ${\rm card}(E_2) = 1360$.
Let $r = (q^{j+1}-q^j)/\log q$ the  $i^{th}$ element
of $E_2$. Then  $T[i] = [q,j,p,l]$
where \dsm{ l = \ell(N_r^+)}
and $p$ is the largest prime $p$ such that $p/\log p < r$.
The superchampions following $N_r^+$ are obtained 
by multiplying it successively by the primes following $p$.
Figure \ref{firstchampions2}
 gives the first values of $T[i]$.
(In the {\sc Maple} program the $T[i]$'s are the elements
of the table {\tt listesuperchE2}).

The superchampions that are not of the form $N_r^+$ for an $r \in E_2$
can easily be obtained from this table.
For instance, the successive values of $\ell(N)$ between $368$ and $626$ are
$368+53=421$, $421+59=480$, $480+61=541$ and $541+67=608$.

Two  elements of $\ce$ can be close.
For instance, the smallest difference between two consecutive elements of
$\ce$ less than $8\cdot 10^9$ is 
\begin{multline*}
\frac{43083996283}{\log 43083996283} - 
\frac{144589^2-144589}{\log 144589}\\
= 1759505912.7146899772-1759505912.7146800938=0.0000098834
\end{multline*} 
and thus, working with $20$ decimal digits is enough to distinguish
the elements of $\ce$.  For any $n$ up to $10^{15}$, Algorithm 2 below
determines the superchampion $N = N_\rho$ as defined in Defintion
\ref{fixN}.

\renewcommand{\algorithmicendwhile}{\textbf{end while}}
\begin{algorithm}
\caption{: computes $N= N_{\rho}$ for a given $n \le 10^{15}$.}
\label{algsuperchamp}
\begin{algorithmic}
\STATE Construct table $T$.
\STATE $i:= $ the largest index such that $T[i].\ell \le n$.
\STATE $\ell':= T[i+1].\ell$,\ $q' = T[i+1].q$,\ $j' = T[i+1].j$.\\
\COMMENT{$r'= ({q'}^{j'}-{q'}^{(j'-1)})/\log q'$ is the smallest element
 in $E_2$ such that $\ell(N_{r'}) > n$}
\STATE $t:= \ell'-q'^{(j'-1)}(q'-1);$ \\
\COMMENT{This is the value $\ell(N)$
 of the superchampion $N$ preceding $N_r^+$}
\IF{$t \le n$}
\STATE $\rho:= r'$
\ELSE
\STATE $n_0$:= $T[i].\ell + nextprime(T[i].p)$;
\WHILE{$n_0 \le n$}
\STATE $p:= nextprime(p);\ n_0:= n_0 + p$
\ENDWHILE
\STATE $\rho:= p/\log p$
\ENDIF\\
\end{algorithmic}
\end{algorithm}

%--------------------------------------------------------------------------
\section{Benefits}\label{benefit}
%--------------------------------------------------------------------------

%--------------------------------------------------------------------------
\subsection{Definition and properties}
%--------------------------------------------------------------------------

\begin{definition}
Let $\rho\in \ce$ and
$N=N_\rho$ (as defined in Definition \ref{fixN}).
If $M$ is a positive integer, from \eqref{super}, we have
$\ell(M)-\rho\log M\ge \ell(N)-\rho\log N$.
We call \emph{benefit} of $M$ the non-negative quantity
\begin{equation}\label{ben}
\ben(M)=\ell(M)-\ell(N)-\rho\log\frac MN \cdot
\end{equation}
Let $M=\prod_p p^{\be_p}$ be the standard factorization of $M$.
We define
\begin{equation}\label{benp}
\benp(M) = \ell(p^{\be_p})-\ell(p^{\al_p})-\rho(\be_p-\al_p)
\log p \ge 0,
\end{equation}
which implies
\begin{equation}\label{benpd}
\ben(M)=\sum_p \benp(M).
\end{equation}
\end{definition}

Geometrically, if we represent $\log M$ in abscissa and $\ell(M)$ in ordinate,
the straight line of slope $\rho$ going through the point $(\log M,\ell(M))$ 
cuts the $y$ axis at the ordinate $y_M=\ell(M)-\rho\log(M)$ and so, the  
benefit is the difference $y_M-y_N$ (see {\sc Fig.} \ref{figurebenefice}).
Note that $\rho=\dfrac{\ell(N')-\ell(N)}{\log N'-\log N}$ with $N=N_\rho$
and $N'=N_\rho^+$.

\begin{figure}
\hspace{-1cm}
\includegraphics[height=7cm]{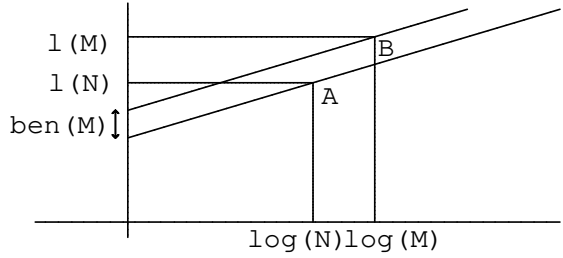}
\caption{$A = (\log N,\ell(N))$ and $B = (\log M,\ell(M))$.}
\label{figurebenefice}
\end{figure}

\begin{lem}\label{bencr}
Let $p\in\cp$, $\al=\al_p=v_p(N)$ and $\ga$ a non-negative integer. Then,
\begin{enumerate}
\item
$\ben(p^\ga N)=\ell(p^{\ga+\al})-\ell(p^\al)-\rho \ga\log p$
is non-decreasing for $\ga\ge 0$ and tends to infinity with $\ga$.
\item
$\ben(N/p^\ga)=\rho \ga\log p+\ell(p^{\al-\ga})-\ell(p^\al)$
is non-decreasing for $0\le \ga\le \al$.
\end{enumerate} 
\end{lem}
\begin{proof}
\begin{enumerate}
\item
If $\ga+\al\ge 1$, we have
\[
\ben(p^{\ga+1}N)-\ben(p^{\ga}N)=
\log p\left(p^\ga \frac{p^{\al+1}-p^{\al}}{\log p}-\rho\right)
\]
which is non-negative from \eqref{Nrho} and tends to
infinity with $\ga$.

If $\al=\ga=0$, we have 
$\ben(pN)-\ben(N)=\log p(p/\log p-\rho)$ which is also non-negative from 
\eqref{Nrho}.
\item
If $\al\ge 2$ and $0\le \ga \le \al-2$, we have
\[
\ben\left(\frac{N}{p^{\ga+1}}\right)-\ben\left(\frac{N}{p^{\ga}}\right)=
\log p\left(\rho-\frac{1}{p^\ga} 
\frac{p^{\al}-p^{\al-1}}{\log p}\right)
\]
which is non-negative from \eqref{Nrho}.

If $\al\ge 1$ and  $\ga=\al-1$,
$$\dsm{\ben\left(\frac{N}{p^{\ga+1}}\right)-
\ben\left(\frac{N}{p^{\ga}}\right)=\log p\left(\rho- \frac{p}{\log
  p}\right)
}$$
yields the same conclusion.
\end{enumerate}
\end{proof}

\begin{lem}
Let $U/V$ be an irreducible fraction such that $V$ divides $N$ (as fixed 
in Definition \ref{fixN}) 
and $U = U_1 U_2$, $V=V_1 V_2$ with $(U_1,U_2) = (V_1,V_2) = 1$.
Then we have
\begin{enumerate}
\item
\begin{equation}\label{ladd}
\ell\left(\frac{UN}{V}\right)-\ell(N)=\ell\left(\frac{U_1N}{V_1}\right)-\ell(N)
+\ell\left(\frac{U_2N}{V_2}\right)-\ell(N).
\end{equation}
\item
\begin{equation}\label{benfrac}
\ben\left(\frac{UN}{V}\right) = \ben\left(\frac{U_1N}{V_1}\right)+\ben\left(
  \frac{U_2N}{V_2}\right)\cdot 
\end{equation}
\end{enumerate}
\end{lem}

\begin{proof}
Observing that a prime $p$ divides at most one of the four
numbers $U_1,U_2,V_1,V_2$ we get \eqref{ladd}.
By the additivity of the logarithm, \eqref{benfrac} follows.
\end{proof}

%--------------------------------------------------------------------------
%\subsection{Bounding the benefits}\label{bobe}
%--------------------------------------------------------------------------
\ni
The following proposition will be useful in the sequel.

\begin{prop}\label{propmajben}
Let $M$ be a positive integer such that $\ell(M)\le n$ (thus, 
from \eqref{pcarg}, $M\le g(n)$ holds). Then,
\[
\ben g(n)\le \ben M+\ell(g(n))-\ell(M)
\]
and
\begin{equation}\label{majben}
\ben g(n)\le \ben g(n)+n-\ell(g(n))\le \ben M+n-\ell(M).
\end{equation}
\end{prop}
 
\begin{proof}
From \eqref{ben}, we have
%\vspace*{-3ex}
\[
\ben g(n)-\ben M = \ell(g(n))-\ell(M)-\rho\log\frac{g(n)}{M}
\le \ell(g(n))-\ell(M)
\]
which implies the first inequality while the second one follows 
from \eqref{lgn}.
\end{proof}

We shall use Proposition \ref{propmajben} to determine an upper bound $B$ 
such that
\begin{equation}\label{benmaxg}
\ben g(n)\le \ben g(n)+n-\ell(g(n))\le B.
\end{equation}

It has been proved in \cite{MNR2} that $B \le x_1$ and 
\begin{equation}\label{benO}
B=\co\left(\frac{x_1}{\log x_1}\right)=\co(\rho),
\end{equation}
and, by the method of \cite{Zak}, it is possible to show that 
$B=o(\rho)$.
The largest quotient $(\ben g(n)+n-\ell(g(n)))/\rho$ that we have 
found up to $n=10^{12}$ is $1.60153$ for $n=45055780$.

%--------------------------------------------------------------------------
\subsection{The benefit of large primes}\label{benlp}
%--------------------------------------------------------------------------
\begin{prop}\label{proplp}
Let $N,B_1,x_1$ and $x_2$ as in Definition \ref{fixN}.
If $M$ is an integer satisfying 
$\ben(M)=\ell(M)-\ell(N)-\rho\log(M/N) < B_1$,
we have
\[
\begin{array}{lllll}
1. &\quad\text{if}& \sqrt{x_1} \le p
&\quad\text{then}\quad&
v_p(M) \le 1\\[1ex]
2. &\quad\text{if}& x_2 \le p < \sqrt{x_1}
&\quad\text{then}\quad&
v_p(M)\le 2.
\end{array}
\]
\end{prop}

\begin{proof}
\mbox{}

\begin{enumerate}
\item
Let us assume that the prime $p$ satisfies $p\ge \sqrt x_1$ and
divides $M$ with exponent $k\ge 2$. 
With \eqref{x1rx2}, we have $p > x_2$ and, 
from \eqref{Nxj}, the exponent
$\al_p$ of $p$ in $N=N_\rho$ is $0$ or $1$. If $\al_p=1$, from
\eqref{benp} and \eqref{Nrho} we have
\begin{eqnarray}\label{benpg1}
\benp(M) &=&  p^k-p-\rho(k-1)\log p = \log p\sum_{i=2}^k 
\Big(\frac{p^i-p^{i-1}}{\log p}-\rho\Big)\notag\\
&\ge& \log p \Big(\frac{p^2-p}{\log p}-\rho\Big)=p^2-p-\rho\log p
\end{eqnarray}
while, if $\al_p=0$,
\begin{eqnarray*}
\benp M &=&  p^k-\rho\, k\log p 
= \log p\, \bigg(\frac{p}{\log p}-\rho + 
\sum_{i=2}^k  \Big(\frac{p^i-p^{i-1}}{\log p}-\rho\Big)\bigg)\\
&\ge& \log p\, \Big(\frac{p^2-p}{\log p}-\rho\Big)=
p^2-p-\rho\log p.
\end{eqnarray*}
So, in both cases, \eqref{benpd} and \eqref{benp} yield $\;\ben M\,
\ge\, \benp M \, \ge\, f(p)\;$ with\\ $f(t)=t^2-t-\rho\log t$.  We
have $f'(t)=2t-1-\rho/t$, $f''(t) > 0$ and, as $x_2 > 2$ holds,
\eqref{defxi} implies
\[
f'(x_2)=2x_2-1-\frac{x_2-1}{\log x_2}\ge 
x_2\left(2-\frac{1}{\log x_2}\right)-1\ge 2\left(2-\frac{1}{\log 2}\right)-1 
> 0
\]
and $f(t)$ is increasing for $t\ge x_2$. Thus, since $p\ge \sqrt{x_1}$,
\[
\ben M \ge f(p)\ge f(\sqrt{x_1})=
x_1-\sqrt{x_1}-\frac{x_1}{\log x_1}\log\sqrt{x_1}
=\frac{x_1}{2}-\sqrt{x_1} \ge B_1
\]
in contradiction with our hypothesis, and 1. is proved.
\item
Let $p$ satisfy $2 < x_2\le p < \sqrt{x_1}$ 
so that, from \eqref{Nxj}, $\al_p=v_p(N)=1$; let us 
assume that $k=v_p(M)\ge 3$;
one would have as in \eqref{benpg1}
\[
\ben M\ge \log p\sum_{i=2}^k 
\left( \frac{p^i-p^{i-1}}{\log p}-\rho\right) \ge p^3-p^2-\rho\log p.
\]
The function $f(t)= t^3-t^2-\rho\log t$ is easily shown to be 
increasing for $t\ge x_2$. From \eqref{defxi},
$f(x_2)= x_2^3-x_2^2-(x_2^2-x_2)$ and thus
\[
 \ben M \ge x_2^3-x_2^2-(x_2^2-x_2)=
x_2(x_2^2-2x_2+1) > x_2^2-2 x_2.
\]
From \eqref{B1}, it follows that $\ben M > B_1$ holds, in contradiction 
with our hypothesis, and 2. is proved.
\end{enumerate}
\end{proof}

%--------------------------------------------------------------------------
\section{Prefixes}\label{prefi}
%--------------------------------------------------------------------------

%--------------------------------------------------------------------------
\subsection{Plain prefixes and suffixes}\label{prefi1}
%--------------------------------------------------------------------------
\begin{definition}
Let $j$ be a positive integer.
\begin{enumerate}
\item
For every positive integer $M$ let us define the fraction
\begin{equation}\label{pjsqrtx1}
\pi^{(j)}(M)= \prod_{p \le p_j} p^{v_p(M)-v_p(N)}
= \prod_{p \le p_j} p^{v_p(M)-\al_p}
\end{equation}
and call $\pi^{(j)}(M)$ the $j$-prefix of $M$. 
\item
We note $\ct_j$, and call it the set of \emph{$j$-prefixes},
the set of fractions
\begin{equation}\label{Tj}
\ct_j=\left\{\de=\prod_{p\le p_j} p^{z_p};\quad z_p \ge -\al_p\right\}.
\end{equation}
\item
For $B' \ge 0$, we define
\begin{equation}\label{deftaujb}
\ct_j(B') = \set{\delta \in \ct_j\; ;\ \ben(N\delta) \le B'}.
\end{equation}
\end{enumerate}
\end{definition}

\begin{definition}
Le $M$ be a positive integer. 
%such that $\ben(M) < B_1$ (defined by \eqref{B1}).
Let us define
\begin{equation}\label{pidef}
\pi(M) = \prod_{p < \sqrt{x_1}} p^{v_p(M)-\al_p}
= \pi^{(j_1)}(M)
\end{equation}
where $p_{j_1}$ is the largest prime less than $\sqrt{x_1}$,
and $\xi(M) = M/(N\pi(M))$. Thus we have
\begin{equation}\label{gpref}
M=N\,\pi(M)\,\,\xi(M).
\end{equation}
$\pi(M)$ will be called the \emph{plain prefix} of $M$, and $\xi(M)$ the
\emph{suffix} of $M$.
\end{definition}

Let us show that, for each $j$ such that $p_j < \sqrt{x_1}$, we have
\begin{equation}\label{benpij}
\ben(N\pi^{(1)}(M))\le \ldots \le \ben(N\pi^{(j)}(M))
\le \ldots \le \ben(N\pi(M)) \le \ben M.
\end{equation}

Indeed, \eqref{benpd} yields $\ben(N\pi^{(j)})=
\sum_{i\le j} \ben_{p_i} M$ and $\ben M = \sum_p \benp M$, which implies
\eqref{benpij}, since, by \eqref{benp}, $\benp M$ is non-negative.

\begin{definition}\label{defsuf}
From now on, we shall note
\begin{equation}\label{defpixi}
\pi^{(j)} = \pi^{(j)}(g(n)),
\qquad \pi = \pi(g(n)), \qquad
\xi = \xi(g(n))
\end{equation}
so that $g(n) = N\pi\xi$ and our work is to compute $\pi$ and $\xi$. 
\end{definition}
Note that $\pi$ and $\xi$ are coprime and \eqref{benfrac} implies
\begin{equation}\label{bengpixi}
\ben g(n)=\ben (N\pi\xi)=\ben (N\pi)+\ben (N\xi).
\end{equation}

%We have the following lemma
\begin{lem}\label{lempref}
Let $j$ be a positive integer and $\de_1 < \de_2$ be two elements of
$\ct_j$ satisfying
\begin{equation}\label{lpij}
\ell\big(\de_2N\big) \le \ell\big(\de_1 N\big).
\end{equation}
Then, $\de_1$ is not the $j$-prefix of $g(n)$\;; in other words, 
$\pi^{(j)} \ne \de_1$.
\end{lem}
\begin{proof}
If $\de_1= \pi^{(j)}$, equation $g(n) = N\pi\xi$ may be written
$g(n)=N \left(\de_1\dfrac{\pi}{\pi^{(j)}}\right)\xi$. 
Set  $M=N \left(\de_2\dfrac{\pi}{\pi^{(j)}}\right)\xi=(\de_2/\de_1) g(n)$. 
%$M=N(\de_2/\pi/\pi^{(j)})\xi=(\de_2/\de_1) g(n)$.
From \eqref{ladd}, \eqref{lpij} and \eqref{lgn}, we get
\begin{eqnarray*}
\ell(M) &=& \ell\big(\de_2N\big)+
\ell \left(N\frac{\pi}{\pi^{(j)}}\right)+\ell(N\xi)-2\ell(N)\\
& \le &\ell\big(\de_1N\big)+\ell\left(N\frac{\pi}{\pi^{(j)}}\right)
+\ell(N\xi)-2\ell(N)= \ell(g(n)) \le n
\end{eqnarray*}
which, from \eqref{pcarg}, implies $M \le g(n)$ and therefore
$\de_2 \le \de_1$, in contradiction with our hypothesis. Note that our 
hypothesis implies $\ben (\de_2 N) < \ben (\de_1 N)$.
\end{proof}

%--------------------------------------------------------------------------
\subsection{Computing plain prefixes}\label{compplain}
%--------------------------------------------------------------------------

Let us suppose that we know an upper bound $B$ 
such that \eqref{benmaxg} holds. 
Then from \eqref{benpij} and \eqref{benmaxg},
for every $j$ such that $p_j < \sqrt x_1$,  $\ben(N\pi^{(j)})\le B$
holds. Let $p_{j_1}$ be the largest prime less than $\sqrt{x_1}$.
Then $\pi = \pi^{(j_1)}(g(n))$ is an element of $\ct_{j_1}(B)$. 

But, we are faced to 2 problems:
First, for the moment, we do not know $B$.
Secondly, for a given value $B'$, the sets $\ct_j(B')$
are too large to be computed efficiently.  

What we can do is the following.
Let $B' < B_1$.
We shall construct two non-decreasing sequences
of sets $\cu_j=\cu_j(B')$ and $\cd_j=\cd_j(B')$ with $\cd_j \subset
\cu_j \subset \ct_j(B')$ satisfying the following property: \emph{
$\cd_j$ contains the $j$-prefix $\pi^{(j)}$ of $g(n)$, provided
that $\ben g(n) \le B'$ holds.}

These sequences are defined by the following induction rule.
The only element of $\ct_0$ is $1$. We set $\cu_0=\cd_0=\{1\}$.
And, for $j \ge 1$,
\begin{itemize}
\item
We define  \dsm{\cu_{j} = 
\setof{\de p_j^{\ga}}{\de \in \cd_{j-1},\; \ga \ge -\al_{p_j}
\text{ and } \ben(N\de p_j^{\ga})\le B'}}.
%We define  $\cu_{j}$ as the subset of $\ct_j(B')$ 
%whose elements are of the form $\de p_{j}^\ga$, with
%$\de \in \cd_{j-1}$.
\item
By lemma \ref{lempref}, if $\de_1 \in \cu_j$ and if there is
a $\de_2$ in $\cu_j$ such that
$\de_1 < \de_2$ and $\ell(N\de_1) \ge \ell(N\de_2)$, 
then $\de_1$ is not the $j$-prefix of $g(n)$.
The set $\cd_j$ is $\cu_j$ from which are removed these $\de_1$'s.
In other words, $\cd_{j}$ will be the pruned set of $\cu_{j}$ (see Section
\ref{merpru}).
\end{itemize}

For each $\de$ in $\cd_{j-1}$, $\de p_{j}^\ga$
belongs to $\cu_{j}$ if
$\ga\ge -\al_{p_{j}}$ and $\ben(N\de p_{j}^\ga)\le B'$ which,  
according to \eqref{benfrac}, can be rewritten as
\begin{equation}\label{benga}
\ben(Np_{j}^\ga) \le B' -\ben(N\de).  
\end{equation}
It results from Lemma \ref{bencr} that $\ben(Np_{j}^\ga)$  
is non-increasing for 
$-\al_{p_{j}}\le \ga\le 0$, non-decreasing for $\ga\ge 0$, 
vanishes for $\ga=0$  and tends  
to infinity with $\ga$. Therefore the 
solutions in $\ga$ of
\eqref{benga} form a finite interval containing $0$. 

Thanks to \eqref{benpij}, by induction on $j$, it can be seen that 
if $\ben g(n)\le B'$, the $j$-prefix $\pi^{(j)}$ of $g(n)$ 
belongs to $\cu_j$ and also to $\cd_j$, by Lemma \ref{lempref}.

We set $\cd(B')=\cd_{j_1}(B')$ and since $\pi=\pi_{j_1}$, 
$\cd(B')$ contains the plain prefix $\pi$ of $g(n)$, 
provided that $\ben g(n) \le B'$ holds.

This construction solves our second problem: at each step of
the induction, the pruning algorithm makes $\cd_j(B')$
smaller than $\cu_j(B')$, and as we progress, $\cd_j(B')$
becomes much smaller than $\ct_j(B')$.

%--------------------------------------------------------------------------
\subsection{Computing $B$, an upper bound for the benefit}\label{upbb}
%--------------------------------------------------------------------------

It remains to find an upper bound
$B$ such that \eqref{benmaxg} holds.
The key is Proposition \ref{propmajben}.
Every $M$ such that $\ell(M) \le n$ gives an upper bound for 
$\ben g(n)+n-\ell(g(n))$:
\[
\ben g(n)\le \ben g(n)+n-\ell(g(n)) \le \ben M+ n -\ell(M).
\]

We choose some $B'$, a provisional value of $B$
satisfying \footnote{
In view of \eqref{benO} and after some experiments, 
our choice is $B'=\rho$ for $2485\le n\le 10^{10}$
while, for greater $n$'s, we take $B'=\rho/2$, and for smaller $n$'s,
$B'=B_1-\vep$ where $\vep$ is some very small positive number.
}  $B' < B_1$ .
Then we compute the set $\cd=\cd(B')$, and by using the prefixes
belonging to  this set we shall construct an integer $M$
to which we apply Proposition \ref{propmajben}.

Let us recall that $p_k$ denotes the greatest prime dividing $N$. To an 
element $\de\in \cd(B')$ and to an integer $\om$, we associate
\[
\de_\om=
\begin{cases}
\de p_{k+1}p_{k+2}\ldots p_{k+\om} & \text{ if } \;\om > 0\\
\de & \text{ if } \; \om=0\\
\de/ (p_{k}p_{k-1}\ldots p_{k+\om+1}) & 
\text{ if }\;  \om < 0 \text{ and } p_{k+\om+1} \ge \sqrt{x_1}. \\
\end{cases}   
\]
From the definition of prefixes, the prime factors of both the numerator
and the denominator of $\de\in \cd(B')$ are smaller than $\sqrt{x_1}$, and 
thus smaller than the primes dividing the numerator or the denominator of 
$\de_\om/\de$.

First, to each $\de\in\cd$, let $\om=\om(\de)$ be the greatest 
integer such that $\ell\big(N\de_{\om}\big)\le n$ (if there is no such 
$\om(\de)$, we just forget this $\de$). We call $\de^{(0)}$ an element 
of $\cd$ which minimizes 
$\ben\big(N\de_{\om}^{(0)}\big)+n-\ell\big(N\de_{\om}^{(0)}\big)$ and  set
$M=N\de_{\om}^{(0)}$. From the construction of $M$, we have $\ell(M)\le n$.
By Proposition \ref{propmajben}, inequality \eqref{benmaxg} 
is satisfied with $B=\ben M+n-\ell(M)$.

If $B\le B'$, we stop and keep $B$; otherwise we start again with $B$
instead of $B'$ to eventually obtain a better bound.

For $n=1000064448$, the value of $\rho$ defined by \eqref{NN'} is equal to
$\rho\approx 12661.7$; the table below displays some values of 
$B'/\rho$ and the corresponding values of Card$(\cd(B'))$ 
and $B/\rho$ given by the above method.
\[
\begin{array}{|c|ccccccccc|}
\hline
B'/\rho & 0 & 0.2 & 0.4 & 0.6 & 0.7 &0.8 & 0.9 & 1 & 1.1\\
\hline
|\cd(B')| & 1 & 11 & 34 & 76 & 109 & 139 & 165 & 194 & 224 \\
\hline
B/\rho & 7.5 & 1.15 & 1.13 & 1.104 & 1.098 & 1.082 & 1.074 & 1.055 & 1.055 \\
\hline
\end{array} 
\]

In this example, if our first choice for $B'$ is $0.6 \rho$, we find 
$B=1.104\rho$. Starting again the algorithm with $B'=1.104 \rho$, we get the
slightly better value $B=1.055 \rho$.

The value of $B$ given by this method is reasonable and less than 10\% 
more than the best possible one: for $n=1000366$, we find
$B\approx 436.04$ while $\ben(g(n)+n-\ell(g(n))\approx 406.1$;
for $n=1000064448$, these two numbers are $13361.6$ and $13285.7$.

%--------------------------------------------------------------------------
\subsection{How many plain prefixes are there?}
%--------------------------------------------------------------------------

Let us denote by $B=B(n)$ the upper bound satisfying \eqref{benmaxg} as 
computed in Section \ref{upbb}. Let us call $\wt{n}$  the integer in the range
$\ell(N)..\ell(N')-1$ such that $B(\wt{n})$ is maximal. 

Let us denote by $\nu=\nu(n)$ the number of possible plain prefixes 
as obtained by the algorithm described in Section \ref{compplain}. Actually, 
this number $\nu$ depends on $B=B(n)$
and we may think that it is a non-decreasing function on $B$ 
so that the maximal number of prefixes 
used to compute $g(m)$ for $\ell(N)\le m < \ell(N')$ 
should be equal to $\nu(\wt{n})$.

For the powers of $10$, the table of {\sc Fig.\,}\ref{figpref}
displays $n$, $\wt{n}$, 
the quotient of the maximal benefit $B(\wt{n})$ by $\rho$, 
the maximal number of plain prefixes $\nu(\wt{n})$ and the exponent
$\log \nu(\wt{n})/\log n$. Note that replacing $\log n$ 
by $\log \wt{n}$ will not change very much this exponent, since 
with the notation of Definition \ref{fixN},
we have $|\wt{n}-n|\le \ell(N')-\ell(N)\le p_{k+1} \lesssim \sqrt{n\log n}$. 

\begin{figure}
\[
\begin{array}{|l|c|l|c|l|}
\hline
&&&\nu(\wt{n})= \# \text{ of }&\text{exponent}=\\
n& \wt{n}&B(\wt{n})/\rho&\text{plain prefixes} &\log \nu(\wt{n})/\log n\\
\hline
10^3&10^3-11&0.9289&14&0.3820\\
10^4&10^4-10&0.8453&19&0.3197\\
10^5&10^5-123&0.8095&22&0.2685\\
10^6&10^6+366&0.9186&51&0.2846\\
10^7&10^7-1269&0.7636&59&0.2530\\
10^8&10^8+639&1.180&85&0.2412\\
10^9&10^9+64448&1.055&212&0.2585\\
10^{10}&10^{10}+88835&0.6884&252&0.2401\\
10^{11}&10^{11}+1007566&0.9278&657&0.2561\\
10^{12}&10^{12}+2043578&1.118&2873&0.2882\\
10^{13}&10^{13}+5276948&0.8331&3805&0.2754\\
10^{14}&10^{14}+17212588&0.6669&7048&0.2749\\
10^{15}&10^{15}-44672895&0.6433&15148&0.2787\\
10^{16}&10^{16}-48912919&0.5077&25977&0.2759\\
10^{17}&10^{17}-426915678&0.6001&72341&0.2858\\
10^{18}&10^{18}+385838833&0.3027&144807&0.2867\\
10^{19}&10^{19}-9639993444&0.2963&170151&0.2753\\
10^{20}&10^{20}+12041967315&0.3218&412151&0.2808\\
%10^{21}&10^{21}-182312257586&0.4383&??&  \\
\hline
\end{array}
\]
\caption{The number of plain prefixes.}
\label{figpref}
\end{figure}

The behaviour of $\nu(\wt{n})$ looks regular and allows to think that
\begin{equation}\label{n0.3}
\nu(\wt{n})=O(n^{0.3}).
\end{equation}

%--------------------------------------------------------------------------
\subsection{For $\ben(M)$ small, prime factors of $\xi(M)$ are large}
%--------------------------------------------------------------------------
If the number $B$ computed as explained in Section \ref{upbb}
is greater than $B_1$ our algorithm fails. Fortunately, we have not yet found
any $n\ge 166$ for which that bad event occurs.

\begin{prop}
If $B$ is computed as explained in Section \ref{upbb} (so that \eqref{benmaxg}
holds) and satisfies  $B <B_1$ (where $B_1$ is defined in \eqref{B1}) then, in
view of \eqref{x1rx2}, there exists a unique real number $t_1$ such that
\begin{equation}\label{x1rhox2}
2 < x_2 < \sqrt{x_1} < \rho = \frac{x_1}{\log x_1}< t_1  < x_1
\end{equation}
and
\begin{equation}\label{deft1}
\rho \log t_1 -t_1 = B.
\end{equation}
Further, if $\ben M \le B$, we have
\begin{enumerate}
\item
If $x_2 \le p < t_1$ then $v_p(M) \ge 1 = v_p(N)$.
\item
If $x_2 \le p < \sqrt{x_1}$ then $v_p(M) \in \set{1,2}$ and $v_p(N)=1$.
\item
If $\sqrt x_1 \le p < t_1$ then $v_p(M) = v_p(N) = 1$. 
\item
If $t_1 \le p < x_1$ then $v_p(M) \in \set{0,1}$ and $v_p(N)=1$.
\item
If $x_1 \le p $ then $v_p(M) \in \set{0,1}$ and $v_p(N)=0$.
\end{enumerate}
\end{prop}

\begin{proof}
The function $f(t) = \rho
\log t -t$ is increasing on $[x_2,\rho]$ and decreasing on
$[\rho,x_1]$.  From \eqref{defxi} and \eqref{B1} we have
\begin{equation*}
f(\rho) > f(x_2) = \frac{x_2^2-x_2}{\log x_2}\log x_2 - x_2 
= x_2^2- 2 x_2 \ge B_1 > B > 0 = f(x_1)
\end{equation*}
which gives the existence and unicity of $t_1$, which belongs to
$(\rho,\, x_1)$. Now we prove points 1,2,3,4,5.
\medskip

Let $p$ be a prime number satisfying $x_2\le p < t_1$.
If $p$ does not divide $M$, from \eqref{benpd} and \eqref{benp} we have
\[\ben M \ge \benp M=\rho\log p-p=f(p) > f(t_1)=B.\]
Since $\ben M\le B$ is supposed to hold, there is a contradiction 
and 1 is proved.

Since we have assumed that $B< B_1$ holds, Proposition \ref{proplp}
may be applied. Point 2. follows from point 1. and from item 2. of
Proposition \ref{proplp}, while point 3. follows from point 1. and
from item 1. of Proposition \ref{proplp}.  Finally, points 4. and 5.
are implied by item 1. of Proposition \ref{proplp}.
\end{proof}

\begin{coro}
Let us assume that $B$ is such that \eqref{benmaxg} and $B < B_1$ hold.
Then the suffix $\xi=\xi(g(n))$ defined in Definition \ref{defsuf}
can be written as
\begin{equation}\label{defsf}
\xi=\xi(g(n)) = 
\frac{p_{i_1}p_{i_2}\dots p_{i_u}}{p_{j_1}p_{j_2}\dots p_{j_v}}
\qquad{u \ge 0,\ v \ge 0}
\end{equation}
where (we recall that $p_k$ is the largest prime factor of $N$)
\begin{equation}\label{hypxi}
2 < x_2 < \sqrt{x_1} < \rho < t_1 \le p_{j_1} < p_{j_2} \dots < p_{j_v }
\le p_k < p_{i_1} < \dots < p_{i_u}.
\end{equation}
\end{coro}

%--------------------------------------------------------------------------
\subsection{Normalized prefix of $g(n)$}\label{normpref}
%--------------------------------------------------------------------------

\begin{definition}
Let $u$ and $v$ be as defined in \eqref{defsf} and $\omega = u-v$.
We define the \emph{normalized suffix} $\sigma$ of $g(n)$
by
\begin{enumerate}
\item
If $\omega \ge 0$ 
\[
\sigma = 
\frac{p_{i_1}\dots p_{i_u}}{p_{j_1}\dots p_{j_v}p_{k+1}\dots p_{k+\omega}}
=\frac{\xi}{p_{k+1}\dots p_{k+\omega}}\cdot
\]
\item
If $\omega <0$, we set $\omega' = -\omega$ and 
\[
\sigma = 
\frac{p_{i_1}\dots p_{i_u}p_k\dots p_{k-\omega'+1}}{p_{j_1}\dots p_{j_v}}=
\xi p_{k}\dots p_{k-\omega'+1}.
\]
\end{enumerate}
The \emph{normalized prefix} $\Pi$ of $g(n)$ is defined by
\begin{equation}\label{PI}
\Pi = \frac{g(n)}{N\sigma}
=
\begin{cases}
\pi p_{k+1}p_{k+2}\dots p_{k+\omega} &\text{ if } \omega \ge 0\\[2.5ex]
\dfrac{\pi}{p_k\dots p_{k-\omega'+1}}&\text{ if } \omega < 0.
\end{cases}
\end{equation}
\end{definition}

\begin{prop}
Let $\sigma$ be the normalized suffix of $g(n)$. Then
\begin{equation*}
\sigma = \frac{Q_1 Q_2\dots Q_s}{q_1 q_2\dots q_s}
\end{equation*}
where $s$ is a non-negative integer with
\begin{enumerate}
\item If $\omega \ge 0$ then $u \le s \le v$ and
\begin{equation}\label{benPI}
\ben(N\Pi) = 
\ben(N\pi) + \sum_{i=1}^{\omega} \ben(Np_{k+i})
=\ben(N\pi)+\sum_{i=1}^{\omega} (p_{k+i}-\rho\log p_{k+i}),
\end{equation}
\begin{equation}\label{Delta}
\ell(\si)=\sum_{i=1}^s (Q_i-q_i)=p_{i_1}+\ldots +p_{i_u}-
(p_{j_1}+\ldots +p_{j_v})-(p_{k+1}+\ldots +p_{k+\om})\ge 0.
\end{equation}
\item If $\omega < 0$ then $v \le s \le u$ and, with
$\omega'=-\omega=v-u$,
we have
\begin{equation}\label{benPI-}
\ben(N\Pi) = 
\ben(N\pi) + \sum_{i=0}^{\omega'-1} \ben\pfrac{N}{p_{k-i}}
=\ben(N\pi)+\sum_{i=0}^{\omega'-1} (\rho\log p_{k-i}-p_{k-i})
\end{equation}
\begin{equation}\label{Delta-}
\ell(\si)=\sum_{i=1}^s (Q_i-q_i)=p_{i_1}+\ldots +p_{i_u}-
(p_{j_1}+\ldots +p_{j_v})+(p_{k}+\ldots +p_{k-\om'+1})\ge 0.
\end{equation}
\end{enumerate}
In both cases we have also
\begin{equation}\label{Qq}
\sqrt{x_1} <\rho < t_1 < q_1 < \dots < q_s \le p_{k+\omega}
< Q_1 < \dots < Q_s.
\end{equation}
\end{prop}

\begin{proof}
If  $u\ge v$ then
$\om=u-v\ge 0$,
\begin{equation}\label{sig}
\si=\frac{p_{i_1}\ldots p_{i_u}}
{p_{j_1}\ldots p_{j_v}p_{k+1}\ldots p_{k+\om}}
=\frac{\xi}{p_{k+1}\ldots p_{k+\om}}\cdot
%%\qtx{and}
%%\xi=p_{k+1}p_{k+2}\ldots p_{k+\om}.
\end{equation}
Since the prime factors $p_{i_1}\ldots p_{i_u}$ of the numerator are distinct
of the prime factors $p_{j_1}\ldots p_{j_v}$ of the denominator, $\si$ 
can be written after simplification
\begin{equation}\label{sigQ}
\si=\frac{Q_1Q_2\ldots Q_s}{q_1q_2\ldots q_s}
\end{equation}
where $v\le s\le u$ and, from \eqref{hypxi}, we have
\begin{equation*}%\label{Qq}
\sqrt{x_1} < \rho < t_1 < q_1 < q_2 < \ldots < q_s \le p_{k+\om} <
Q_1 < Q_2 < \ldots < Q_s
\end{equation*}
which is \eqref{Qq}.
From \eqref{benfrac} we get \eqref{benPI} while \eqref{Delta} follows 
from \eqref{sig} and \eqref{sigQ}. 
\medskip

Similarly, if $u < v$ holds, $\om'=v-u > 0$.
So, $\om'\le v$, and from \eqref{hypxi}, 
$p_{k-\om'+1}\ge p_{k-v+1}\ge p_{j_1}>t_1$; 
\eqref{sig} and \eqref{sigQ} become
\begin{equation}\label{sig-}
\si= %%\xi\, p_{k}p_{k-1}\ldots p_{k-\om'+1}=
\frac{p_{i_1}\ldots p_{i_u}p_{k}\ldots p_{k-\om'+1}}{p_{j_1}\ldots p_{j_v}}=
\frac{Q_1\ldots Q_s}{q_1\ldots q_s}
\end{equation}
where $u\le s\le v$ and we have 
\begin{equation}\label{Qq-}
\sqrt{x_1} < \rho < t_1 < q_1 <  \ldots < q_s \le p_{k-\om'} =p_{k+\om} <
Q_1 <  \ldots < Q_s.
\end{equation}
which is again \eqref{Qq}.
 
By definition, any prime factor of $\pi$ is smaller than $\sqrt{x_1}$. 
Therefore, by \eqref{Qq-}, $p_{k-\om'+1}$ 
is greater than any prime factor of $\pi$, \eqref{benfrac} can be applied 
and \eqref{benPI} becomes \eqref{benPI-} 
while \eqref{Delta} becomes \eqref{Delta-}.
\end{proof}

The value of the parameter $\om$ can be computed from the following 
proposition.
It is convenient to set $S_\om=\sum_{i=1}^\om p_{k+i}$ (for $\om \ge 0$)
and $S_\om=-\sum_{i=0}^{-\om-1} p_{k-i}$ (for $\om < 0$). In both cases, 
from \eqref{ladd}, we have 
\begin{equation}\label{Som}
S_\om=\ell(N\Pi)-\ell(N\pi).
\end{equation}
\begin{prop}\label{propPi}
The relative integer $\om$ which determines the normalized prefix $\Pi$ 
of $g(n)$ (cf. \eqref{PI}) satisfies the following inequalities:
\begin{equation}\label{calculom}
n-\ell(N\pi)-\frac{B}{1-\rho/t_1} \le 
n-\ell(N\pi)-\frac{B-\ben (N\Pi)}{1-\rho/t_1} \le S_\om\le n-\ell(N\pi)
\end{equation}
where $\pi$ is the prefix of $g(n)$ and
$B$ and $t_1$ satisfy \eqref{benmaxg} and \eqref{deft1}.
\end{prop}

\begin{proof}
Let us prove Proposition \ref{propPi} for $\om\ge 0$; 
the case $\om < 0$ is similar.
From \eqref{sigQ}, \eqref{Qq} and \eqref{Delta}, Lemma \ref{lemme1} (i) yields 
\begin{equation}\label{sigG}
1\le \si \le \exp\left(\frac{\ell(\si)}{t_1}\right).
\end{equation}
From \eqref{defsf} and \eqref{Delta}, we have
\begin{equation}\label{lsig}
\ell(N\xi)-\ell(N)= p_{i_1}+\ldots +p_{i_u}-(p_{j_1}+\ldots +p_{j_v})=
\ell(\si)+S_\om.
\end{equation}
So, we get successively
\begin{eqnarray*}
\ben(N\xi) &=& 
\ell(N\xi)-\ell(N)-\rho\log \xi\quad\text{by \eqref{ben}}\notag\\
&=&\ell(\si)+\sum_{i=1}^\om (p_{k+i}-\rho\log p_{k+i})-\rho\log \si 
\quad\text{by \eqref{sig}}\notag\\
& \ge & \ell(\si)+\sum_{i=1}^\om (p_{k+i}-\rho\log p_{k+i})
-\frac{\rho \ell(\si)}{t_1}\quad\text{by \eqref{sigG}}\notag\\
&=&  \ell(\si)\left(1-\frac{\rho}{t_1}\right)
+\ben(N\Pi)-\ben(N\pi)\quad\text{by \eqref{benPI}}.%\label{bensig}
\end{eqnarray*}
From \eqref{Delta}, we have $\ell(\si)\ge 0$.
Since, from \eqref{Qq}, $\rho < t_1$ holds, the above result together with
\eqref{bengpixi}, \eqref{benmaxg} and \eqref{lgn} implies that

\begin{multline}\label{Delmaj}
0\le \ell(\si) \le 
\frac{\ben(N\xi)-\ben(N\Pi)+\ben(N\pi)}{1-\rho/t_1}
=\frac{\ben g(n)-\ben(N\Pi)}{1-\rho/t_1}\\
\le \!\frac{B-\ben(N\Pi)-n+\ell(g(n))}{1-\rho/t_1}
\le \frac{B-\ben(N\Pi)}{1-\rho/t_1}-(n-\ell(g(n))).
\end{multline}
Now, from \eqref{ladd}, and \eqref{lsig}, we get
\begin{equation}\label{616}
\ell(g(n))=\ell(N\pi\xi)=\ell(N\pi)+\ell(N\xi)-\ell(N)=
\ell(N\pi)+\ell(\si)+S_\om.
\end{equation}
Further, since
\begin{equation}\label{l(sig)}
n-\ell(N\pi)=\ell(g(n))-\ell(N\pi)+n-\ell(g(n))=\ell(\si)+
S_\om+n-\ell(g(n)),
\end{equation}
we get from \eqref{Delmaj} and \eqref{lgn}
\begin{equation}\label{calculom2}
n-\ell(N\pi)-\frac{B-\ben(N\Pi)}{1-\rho/t_1} \le
S_\om\le n-\ell(N\pi)
\end{equation}
and \eqref{calculom} follows, since $\ben(N\Pi)\ge 0$. Note that 
\eqref{calculom2} implies
\begin{equation}\label{benPIB}
\ben(N\Pi)\le B.
\end{equation}
\end{proof}

%--------------------------------------------------------------------------
\subsection{Computing possible normalized prefixes}\label{compnorm}
%--------------------------------------------------------------------------

In Section \ref{compplain}, we have computed $B$ such that 
\eqref{benmaxg} holds
and a set $\cd=\cd(B)$ containing the plain prefix $\pi$ of $g(n)$.
By construction, we know that any prime factor of $\pi\in\cd$ 
is smaller than $\sqrt{x_1}$ and thus, from \eqref{x1rhox2}, smaller 
than $t_1$.

\begin{definition}
We call \emph{possible normalized prefix} a positive rational number 
$\Pih=\Pih(\pih,\om)$ 
of the form $\Pih=\pih p_{k+1}\ldots p_{k+\om}$ (with $\om\ge 0$) or
$\Pih=\pih/(p_k\ldots p_{k+\om+1})$ (with $\om < 0$), 
where $\pih\in\cd(B)$ is a plain prefix, and satisfying
\begin{equation}\label{pkt1h}
p_{k+\om+1}\ge t_1
\end{equation}
and
\begin{equation}\label{Som1}
n-\ell(N\pih)-\frac{B}{1-\rho/t_1} \le 
n-\ell(N\pih)-\frac{B-\ben (N\Pih)}{1-\rho/t_1} \le S_\om\le n-\ell(N\pih)
\end{equation}
with $S_\om=\sum_{i=1}^\om p_{k+i}$ (if $\om \ge 0$)
and $S_\om=-\sum_{i=0}^{-\om-1} p_{k-i}$ (if $\om < 0$).
\end{definition}

Let us denote by $\cn$ the set of possible normalized prefixes; $\cn$ has 
been defined in such a way that the normalized prefix $\Pi$ of $g(n)$
belongs to $\cn$. Indeed, from \eqref{PI}, $\Pi$ 
has the suitable form, the plain prefix $\pi$ of $g(n)$ 
belongs to $\cd(B)$, \eqref{Som1} is satisfied by Proposition \ref{propPi} 
and \eqref{pkt1h} by \eqref{Qq}.

Let us observe that, if $\om$ increases by $1$, by \eqref{Qq},
$S_\om$ increases by at least $t_1$. In practice,
$1-\rho/t_1$ is close to $1$ and $B$ is much smaller than $t_1$ so that
for most of the $\pih$'s there is no solution to \eqref{Som1} 
and there are few possible normalized prefixes.
For $n$ in the range $[998001,1000000]$, the number of 
possible normalized prefixes is $1$ (resp. $2$ or $3$) for 1439 values (resp. 
547 or 94). For instance, for $n=998555$, the three possible 
normalized prefixes are $1,43/41,11/10$.

Finally, for a reason given in the next section,
for every $\Pih\in\cn$, we check that the following inequality holds:
\begin{equation}\label{pkomt1}
p_{k+\om+1}-(n-\ell(N\Pih))\ge \sqrt{x_1}.
\end{equation}
This inequality seems reasonable, since, from \eqref{pkt1h}, we have
$p_{k+\om+1}\ge t_1$ with $t_1$  close to $x_1$, and, from
\eqref{Som1}, $n-\ell(N\Pih)=n-\ell(N\pih)-S_\om \le B/(1-\rho/t_1)$ 
which is much smaller than $x_1$. We have not found any counterexample to
\eqref{pkomt1}.

%--------------------------------------------------------------------------
\subsection{The heart of the algorithm}\label{heart}
%--------------------------------------------------------------------------

We have now a list $\cn$ of possible
normalized prefixes containing the normalized prefix $\Pi$ of
$g(n)$. For $\Pih=\Pih(\pih,\om)\in \cn$ let us introduce
\begin{equation}\label{gPIn}
g(\Pih,n)=
N\Pih G(p_{k+\om},n-\ell(N\Pih))=
N\Pih \frac{Q_1Q_2\ldots Q_s}{q_1q_2\ldots q_s}
\end{equation}
where $G(p_{k+\om},n-\ell(N\Pih))=\dfrac{Q_1Q_2\ldots Q_s}{q_1q_2\ldots q_s}$ 
is defined by \eqref{G1}. 
We shall use the following proposition to compute $g(n)$.
\begin{prop}
The following formula gives the value of $g(n)$:
\begin{equation}\label{gmaxPI}
g(n)=\max_{\Pih\in \cn}\; g(\Pih,n)=
\max_{\Pih\in \cn}\; N\Pih G(p_{k+\om},n-\ell(N\Pih)).
\end{equation}
\end{prop}

\begin{proof}
Note that 
\eqref{qsQ} and \eqref{Qimqi} imply either $s=0$ or
the smallest prime factor $q_s$ of $G(p_{k+\om},n-\ell(N\Pih))$ 
satisfies $p_{k+\om+1}-q_s\le n-\ell(N\Pih)$ which, from \eqref{pkomt1}, 
implies $q_s\ge \sqrt{x_1}$ and thus, the 
prime factors of $\pi$ and those of $G(p_{k+\om},n-\ell(N\Pih))$ are distinct.
Therefore, for any $\Pih=\Pih(\pih,\om) \in\cn$ with $\om \ge 0$,
we get from \eqref{gPIn}, \eqref{ladd} and \eqref{lGm} 
\begin{eqnarray*}
\ell(g(\Pih,n)) &=& \ell(N\pih)+
\ell\left(N\frac{p_{k+1}\ldots p_{k+\om}Q_1\ldots Q_s}{q_1\ldots q_s}\right)
-\ell(N)\\
&=&\ell(N\pih)+\sum_{i=1}^\om p_{k+i}+\sum_{i=1}^s (Q_i-q_i)\\
&=& \ell(N\Pih)+\ell(G(p_{k+\om},n-\ell(N\Pih)))\\
&\le& \ell(N\Pih)+n-\ell(N\Pih)=n.
\end{eqnarray*} 
Inequality $\ell(g(\Pih,n))\le n$ can be proved similarly in the case
$\om < 0$.
\bigskip

\ni
Since $\ell(g(\Pih,n))\le n$ holds, \eqref{pcarg} implies for all $\Pih\in \cn$
\begin{equation}\label{gPIsup}
g(\Pih,n)\le g(n).
\end{equation}
From \eqref{PI}, we get $g(n)=N\Pi \si$ where 
$\Pi$ is the normalized prefix of $g(n)$. Now, if $\om \ge 0$,
from \eqref{Delta}, \eqref{616}, \eqref{PI} 
and \eqref{lgn}, we have 
\begin{eqnarray}\label{msi}
\ell(\si)=\sum_{i=1}^s (Q_i-q_i) &=& \ell(g(n))-\ell(N\pi)
-\sum_{i=1}^\om p_{k+i}\notag\\
&=&\ell(g(n))-\ell(N\Pi) \le n-\ell(N\Pi)
\end{eqnarray}
($\ell(\si)\le n-\ell(N\Pi)$ still holds for $\om < 0$). Therefore, in 
view of \eqref{Qq}  and of 
Definition \eqref{G1} of function $G$, we have
\begin{equation}\label{gPIG}
g(n)=N\Pi\si \le N\Pi G(p_{k+\om},n-\ell(N\Pi))=g(\Pi,n).
\end{equation}
Since $\Pi\in \cn$, \eqref{gPIG} and \eqref{gPIsup} prove \eqref{gmaxPI}.
\end{proof}

%--------------------------------------------------------------------------
\subsection{The fight of normalized prefixes}\label{fight}
%--------------------------------------------------------------------------

Let $\Pih_1$ and $\Pih_2$ two normalized prefixes.
By using Inequalities \eqref{Gineq} below, it is sometimes possible 
to eliminate $\Pih_1$ or $\Pih_2$.

Indeed, from \eqref{Gineq}, we deduce a lower and 
an upper bound for $g(\Pih,n)$ (defined in \eqref{gPIn}):
\[g'(\Pih,n)\le   g(\Pih,n) \le g''(\Pih,n).\]
If, for instance, $g''(\Pih_1,n) < g'(\Pih_2,n)$ holds, then clearly $\Pih_1$ 
cannot compete in \eqref{gmaxPI} to be the maximum. 

By this simple trick, it is possible to shorten the list $\cn$ of 
normalized prefixes. For instance, for $n=10^{15}$, the number of 
normalized prefixes is reduced from 9 to 1, while, for $n=10^{15}+
123850000$, it is reduced from 37 to 2.

%--------------------------------------------------------------------------
\section{A first way to compute $G(p_k,m)$}\label{grandpe}
%--------------------------------------------------------------------------

%--------------------------------------------------------------------------
\subsection{Function $G$}\label{subG}
%--------------------------------------------------------------------------

In this section, we study the function $G$ introduced in \eqref{G1}.
First, for $k\ge 3$ and $0\le m \le p_{k+1}-3$, we consider the set
\begin{equation}\label{cG}
\cg(p_k,m)=\left\{ F=\frac{Q_1Q_2\ldots Q_s}{q_1q_2\ldots q_s}; \quad
\ell(F)=\sum_{i=1}^s (Q_i-q_i) \le m,\;\; s\ge 0\right\}
\end{equation}
where the primes $Q_1,Q_2,\ldots,Q_s,q_1,q_2,\ldots,q_s$ satisfy \eqref{qsQ}.

The parameter $s=s(F)$ in \eqref{cG} is called the number of factors 
of the fraction $F$. If $s=0$, we set $F=1$ and $\ell(F)=0$ so that 
$\cg(p_k,m)$ contains $1$ and is never empty. The definition 
\eqref{G1} can be rewritten as
\begin{equation}\label{G}
G(p_k,m)=\max_{F\in \cg(p_k,m)} F.
\end{equation}
Obviously, $G(p_k,m)$ is non-decreasing on $m$ and $G(p_k,2m+1)=G(p_k,2m)$.
Note that the maximum in \eqref{G} is unique (from the unicity of the standard
factorization into primes). 
It follows from \eqref{qsQ} that, if $0\le m < p_{k+1}-p_k$, the set 
$\cg(p_k,m)$ contains only $1$, and therefore, 
\begin{equation}\label{Gpkm1}
0\le m < p_{k+1}-p_k\quad \LR \quad G(p_k,m)=1. 
\end{equation}

\begin{prop}\label{propgpk}
\begin{enumerate}
\item
Let $q$ be the smallest prime satisfying $q\ge p_{k+1}-m$. The following 
inequality holds
\begin{equation}\label{Gineq}
\frac{p_{k+1}}{q}\le G(p_k,m)\le \frac{p_{k+1}}{p_{k+1}-m}.
\end{equation}
Note that if $q=p_{k+1}-m$ is prime, then \eqref{Gineq} yields the
exact value of $G(p_k,m)$.
\item
Now, let $F=\dfrac{Q_1Q_2\ldots Q_s}{q_1q_2\ldots q_s}$ be any element 
of $\cg(p_k,m)$; we have
\begin{equation}\label{Gmin1}
G(p_k,m)\ge F \ge 1+\frac{\ell(F)}{p_k}\cdot
\end{equation}
\end{enumerate}
\end{prop}

\begin{proof}
The lower bound in \eqref{Gineq} is obvious. Let us prove the upper bound.
If $0\le m < p_{k+1}-p_k$, the upper bound of \eqref{Gineq} follows 
by \eqref{Gpkm1}.
If $m \ge  p_{k+1}-p_k$, $\dfrac{p_{k+1}}{p_k}\in \cg(p_k,m)$ and thus
$G(p_k,m)\ge \dfrac{p_{k+1}}{p_k} > 1$.
Moreover, with the notation 
\eqref{cG}, if $G(p_k,m)=F=\dfrac{Q_1Q_2\ldots Q_s}{q_1q_2\ldots q_s}$, 
we have $s\ge 1$ and Lemma \ref{lemme1} (ii) implies
\begin{equation}\label{Gmaj}
G(p_k,m)\le \frac{Q_s}{Q_s-\ell(F)}\le \frac{Q_s}{Q_s-m}
\le \frac{p_{k+1}}{p_{k+1}-m}
\end{equation}
where the last inequality follows from \eqref{qsQ} and the
decrease of $t \mapsto t/(t-m)$.

Let us now prove   \eqref{Gmin1}.
This inequality holds if $\ell(F)=0$ (i.e., $F=1$ and $s=0$). If $s > 0$,
from \eqref{qsQ}, we get
\[\frac{Q_i}{q_i}=1+\frac{Q_i-q_i}{q_i}\ge 1+\frac{Q_i-q_i}{p_k},
\qquad i=1,2,\ldots,s\]
and
\[F=\prod_{i=1}^s \frac{Q_i}{q_i}\ge \prod_{i=1}^s\left(1+\frac{Q_i-q_i}{p_k}
\right) \ge 1+\frac{\sum_{i=1}^s(Q_i-q_i)}{p_k}=1+\frac{\ell(F)}{p_k}\cdot\]
\end{proof}

%--------------------------------------------------------------------------
\subsection{Function  $H$}
%--------------------------------------------------------------------------

Let $M\le p_{k+1}-3$; we want to calculate $G(p_k,m)$ for $0\le m \le M$.
Let us introduce a family of consecutive primes $P_0 < P_1 < \ldots < P_K=p_k
< P_{K+1} < \ldots < P_R < P_{R+1}$ (so that $P_i=p_{k+i-K}$ for 
$0\le i\le R+1$)
with the properties
\begin{equation}\label{pRpR'}
P_{R+1}-P_K > M,\quad R\ge K+1,\quad P_{K+1}-P_{0} > M, \quad P_1\ge 3.
\end{equation}
It follows from \eqref{cG} and \eqref{qsQ} that 
the prime factors $Q_1, \ldots, Q_s,q_1,\ldots, q_s$ of any element of
$\cg(p_k,m)=\cg(P_K,m)$ should satisfy
\begin{equation}\label{qRQ}
P_{1} \le q_s < \ldots < q_1 \le P_K = p_k < P_{K+1} \le Q_1 < 
\ldots < Q_s\le P_R.
\end{equation}
Of course, in \eqref{pRpR'} we may choose $P_R$ (resp.~$P_1$) as small 
(resp.~large) as possible, but it is not an obligation.

Let us denote by $Q'_{1},Q'_{2},\ldots, Q'_{R-K-s}$ the primes among
$P_{K+1},\ldots, P_R$ which are different of $Q_{1},\ldots, Q_{s}$; 
we have
\begin{equation}\label{Q'}
Q'_{1}+Q'_{2}+\ldots+ Q'_{R-K-s}=P_{K+1}+\ldots +P_{R}-
\left(Q_{1}+\ldots +Q_{s}\right)
\end{equation}
and \eqref{G} becomes
\begin{equation}\label{GQ'}
G(P_K,m)=\max \frac{P_{K+1}P_{K+2}\ldots P_{R}}
{Q'_{1}\ldots Q'_{R-K-s}q_{1}\ldots q_{s}}=
\frac{P_{K+1}P_{K+2}\ldots P_{R}}
{\min(q'_{1}\ldots q'_{R-K})}
\end{equation} 
where the minimum is taken over all the subsets 
$\{q'_{1},q'_{2}, \ldots, q'_{R-K}\}$ of $R-K$ elements of 
$\{P_{1},\ldots,P_R\}$ satisfying from \eqref{Qimqi} and \eqref{Q'}
\begin{eqnarray}\label{7hypH}
q'_{1}+q'_{2}+\ldots +q'_{R-K}
&=& Q'_{1}+Q'_{2}+\ldots + Q'_{R-K-s}+ q_1+q_2+\ldots +q_s\notag\\
&=& P_{K+1}+ P_{K+2}+\ldots + P_{R} -\sum_{i=1}^s (Q_i-q_i)\notag\\
&\ge&  P_{K+1}+ P_{K+2}+\ldots + P_{R} -m.
\end{eqnarray}
(Note that, from \eqref{pRpR'}, $R-K\ge 1$ holds).

\begin{definition}
For $1 \le r \le R$, $1\le j\le \min(r,R-K) \le R$ and
$m\ge 0$, we define
\begin{equation}\label{H}
H(j,P_r;m)=\min (q'_{1} q'_{2} \ldots q'_{j}) 
\end{equation}
where the minimum is taken over the $j$-uples of primes 
$(q'_{1},q'_{2}, \ldots, q'_{j})$ satisfying 
\begin{equation}\label{R'q'R}
P_{1} \le q'_1 < q'_2 < \ldots < q'_j \le P_r 
\end{equation}
and
\begin{equation}\label{q'mj}
q'_1 +q'_2 + \ldots + q'_j\ge P_{K+1}+P_{K+2}+\ldots + P_{K+j} -m.
\end{equation}
If there is no $(q'_{1},q'_{2}, \ldots, q'_{j})$ such that 
\eqref{R'q'R} and \eqref{q'mj} hold, we set
\begin{equation}\label{Hinfini}
H(j,P_r;m)=+\iy.
\end{equation}
\end{definition}

By the unicity of the standard factorization into primes, the minimum 
in \eqref{H} is unique and \eqref{GQ'}  and \eqref{H} yield
\begin{equation}\label{GH}
G(p_k,m)=G(P_K,m)=\frac{P_{K+1}P_{K+2}\ldots P_{R}}{H(R-K,P_R;m)}\cdot
\end{equation}
For $j=R-K$ and $r=R$, the $j$-uple $q'_{1},q'_{2}, \ldots, q'_{j}$ defined by
$q'_i=P_{K+i}$ satisfies \eqref{R'q'R} and \eqref{q'mj} for all $m\ge 0$;
so, $H(R-K,P_R;m)$ is at most $P_{K+1}P_{K+2}\ldots P_{R}$ and is finite.

%--------------------------------------------------------------------------
\subsection{A combinatorial algorithm to compute  $H$ 
and $G$}\label{comHG}
%--------------------------------------------------------------------------

\begin{definition}
For every integers $(r,j)$, $1 \le r \le R$ and $1 \le j \le R-K$, we define
\begin{multline}
m_j(P_r) = \\
\begin{cases}
P_{K+1}+P_{K+2}+\ldots +P_{K+j} -
(P_{r}+P_{r-1}+\ldots + P_{r-j+1}) &\text{ if } j \le r\\
+\infty &\text{ if } j > r.
\end{cases}
\end{multline}
\end{definition}

\begin{rem}
If $j\ge r+1$, \eqref{R'q'R} cannot be satisfied and,
from \eqref{Hinfini}, $H(j,P_r;m)=+\iy$ for 
all $m\ge 0$. 
If $j\le r$, from \eqref{q'mj}, it follows that, if $m\ge m_j(P_r)$, 
$H(j,P_r;m)\le P_{r}P_{r-1}\ldots P_{r-j+1}$ while, 
by \eqref{Hinfini}, if $m < m_j(P_r)$, $H(j,P_r;m)=+\iy$.
So that, in all cases, if $m < m_j(P_r)$, $H(j,P_r;m)=+\iy$.
\end{rem}

Note that, for $j$ fixed, $m_j(P_r)$ is non-increasing on $r$ since,
for $j\le r$,
\begin{equation}\label{mjr}
m_j(P_{r-1})-m_j(P_r)=
\begin{cases}
+\iy & \text{if } j=r\\
P_r-P_{r-j} > 0 & \text{if } 1\le j\le r-1,
\end{cases}
\end{equation}
and, for $j \ge r+1$, $m_j(P_{r-1})$ and $m_j(P_r)$ are both 
$+\infty$.
On the other hand, if $j\le \min(r,R-K)$ for every $m$ such that
\begin{equation*}\label{Mj}
m\ge M_j(P_r)=P_{K+1}+P_{K+2}+\ldots +P_{K+j}-
(P_{1}+P_{2}+\ldots + P_{j}),
\end{equation*}
$H(j,P_r;m)$ is equal to $P_{1}P_{2}\ldots P_{j}$.

\begin{prop}
For $j=1$, from \eqref{H}, \eqref{R'q'R} and \eqref{q'mj}, we have
\begin{equation}\label{H1}
H(1,P_r;m)=
\begin{cases}
P_{1} & \text{if } m\ge M_1(P_r)=P_{K+1}-P_{1}\\
\ldots&    \\
P_{i} & \text{if } 1 < i < r \text{ and } 
P_{K+1}-P_{i}\le m< P_{K+1}-P_{i-1}\\
\ldots&    \\
P_{r} & \text{if } m_1(P_r)=P_{K+1}-P_{r}\le m< P_{K+1}-P_{r-1}\\
\iy & \text{if } m < m_1(P_r)=P_{K+1}-P_{r}.
\end{cases}
\end{equation}
Further, we have the induction formula:
\begin{equation}\label{Hind}
H(j,P_r;m)=\min\left( H(j,P_{r-1};m), 
P_{r} H(j-1,P_{r-1};m-P_{K+j}+P_r) \right).
\end{equation}
\end{prop}

\begin{proof}
The calculation of $H(1,P_r;m)$ is easy. Let us show the induction formula
\eqref{Hind}.
Either $P_r$ does not divide $H(j,P_r;m)$ and 
$H(j,P_r;m)\!=\!H(j$, $P_{r-1};m)$ or $P_r=q'_j$ is the greatest 
prime factor of $H(j,P_r;m)=q_1'q_2'\ldots q_j'$ and from \eqref{q'mj}, we get 
$q'_1+\ldots+q'_{j-1}\ge P_{K+1}+\ldots+P_{K+j-1}-(m-P_{K+j}+P_r)$.
\end{proof}

Note that if $m\ge m_j(P_r)$, $m-P_{K+j}+P_r\ge m_{j-1}(P_{r-1})$
since $m_j(P_r)=m_{j-1}(P_{r-1})+P_{K+j}-P_r$ so that 
$H(j,P_r;m)$ and $H(j-1,P_{r-1};m-P_{K+j}+P_r)$ are 
simultaneously finite or infinite. \eqref{mjr} implies that
$m_j(P_r)$ and $m_{j}(P_{r-1})$ are both infinite or
$m_{j}(P_{r-1}) > m_j(P_r)$.
For $m_j(P_r)\le m <m_{j}(P_{r-1})$, \eqref{Hind} reduces to 
\begin{equation}\label{Hindred}
H(j,P_r;m)=P_{r} H(j-1,P_{r-1};m-P_{K+j}+P_r)
\end{equation}
while, for $m\ge m_j(P_{r-1})$, the three  
values of the function $H$ in \eqref{Hind} are finite.

From \eqref{H1}, we may remark that, if we set
\begin{equation}\label{H0}
H(0,P_r;m)=1\quad \text{ for all } r\ge 1 \text{ and } m\ge 0,
\end{equation}
the induction formula \eqref{Hind} still holds for $j=1$.

In view of \eqref{GH}, for $1\le r\le R$, $1\le j\le \min(r,R-K)$ and 
$m_j(P_r)\le m \le M$, we calculate $H(j,P_r;m)$ by induction,
using for that \eqref{H0}, \eqref{Hind} and \eqref{Hindred}. 
If $K+2 \le r\le R$, it is useless to calculate 
$H(j,P_r;m)$ for $j < r-K$.

Finally, after getting the value of $H(R-K,P_R;m)$ for 
$m_{R-K}(P_R)=0\le m\le M$, we compute $G(p_k,m)$ by \eqref{GH}.

%--------------------------------------------------------------------------
\subsection{Bounding the largest prime}\label{bolapr}
%--------------------------------------------------------------------------

It turns out that the largest prime used in the computation of $G(p_k,m)$
for $0\le m\le M$ is much smaller than $P_R$ defined in \eqref{pRpR'}.
For instance, for $p_k=P_K=150989$ and $M=5000$, $R$ defined by \eqref{pRpR'}
is at least equal to $K+425$ while only the primes up to 
$p_{k+5}=P_{K+5}=151027$ are used.

So, the idea is to replace $R$ by a smaller number $\Rh$, $K+1\le \Rh < R$, 
and to calculate by induction $H(\Rh-K,P_{\Rh};m)$ instead of 
$H(R-K,P_R;m)$. We get the fraction
$\Fh=\dfrac{P_{K+1}P_{K+2}\ldots P_{\Rh}}{H(\Rh-K,P_{\Rh};m)}$ 
which satisfies
\dsm{\Fh \le G(p_k,m).} Now we have the following lemma.

\begin{lem}\label{lemFchap}
Let $F$ be a real number satisfying  
$1 < F \le G(p_k,m) = \dfrac{Q_1 Q_2 \dots Q_s}{q_1 q_2 \dots q_s}$.
Then, the largest prime factor $Q_s$ of the numerator of $G(p_k,m)$
is bounded above by
\begin{equation}\label{Qsm}
Q_s \le \min\left(p_k+m, \frac{mF}{F-1}\right)\cdot
\end{equation}
\end{lem}

\begin{proof}
Using Lemma 1 and \eqref{lGm}, we write
\[
F\le G(p_k,m)=\frac{Q_1 Q_2 \dots Q_s}{q_1 q_2\dots q_s}
\le\frac{Q_s}{Q_s - \ell(G(p_k,m))}\le\frac{Q_s}{Q_s - m}
\]
which yields \dsm{Q_s \le \frac{mF}{F-1}}. On the other hand,
Inequality \eqref{qsQ} together with \eqref{Qimqi} implies
$Q_s - p_k \le Q_s - q_s \le m$ which completes the proof of
\eqref{Qsm}.
\end{proof}

If $\Fh=\dfrac{P_{K+1}P_{K+2}\ldots P_{\Rh}}{H(\Rh-K,P_{\Rh};m)} > 1$
and if $P_{\Rh} >  \min\left(P_K+m, \dfrac{m\Fh}{\Fh-1}\right)$, it follows 
from Lemma \ref{lemFchap} that 
$G(p_k,m)=\Fh$. If not, we start again by choosing a new value of  
$P_{\Rh}$  greater than 
$\min\left(P_K+m,\dfrac{m\Fh}{\Fh-1}\right)$.
Actually, Inequality \eqref{Qsm} gives a reasonably good 
upper bound for $Q_s$.
In the program, our first choice is $\Rh=K+10$.

%--------------------------------------------------------------------------
\subsection{Conclusion}\label{concG}
%--------------------------------------------------------------------------

The running time of the algorithm described in sections \ref{comHG} and
\ref{bolapr} to calculate $G(p,m)$ for $m\le M$
grows about quadratically in $M$, so, it is rather slow when $M$ 
is large.

For instance, the computation of $g(10^{15}-741281)$ leads to the evaluation of
$G(p,688930)$ for $p=192678883$, and this is not doable by the above 
combinatorial algorithm.

In the next section, we present a faster algorithm to compute $G(p_k,m)$ when 
$m$ is large, but which does not work for small $m$'s so that the 
two algorithms are complementary.

%--------------------------------------------------------------------------
\section{Computation of $G(p_k,m)$ 
for $m$ large}\label{Glarge}
%--------------------------------------------------------------------------

The algorithm described in this section starts from the following 
two facts:
\begin{itemize}
\item
if 
$G(p_k,m)=\dfrac{Q_1Q_2\ldots Q_s}{q_1q_2\ldots q_s}$ and $m$ is large,
the least prime factor $q_s$ of the denominator is close to $p_{k+1}-m$
while all the other primes $Q_1,\ldots, Q_s,q_1,\ldots,$ $q_{s-1}$ 
are close to $p_k$. More precisely, $G(p_k,m)$ is equal to
$\dfrac{p_{k+1}}{q_s}G(p_{k+1},d)$ where $d=m-p_{k+1}+q_s$ is small. 

Note that when $m$ is small $G(p_k,m)$ is not always equal to\\
$\dfrac{p_{k+1}}{q_s}G(p_{k+1},m-p_{k+1}+q_s)$. For instance,
$G(103,22)=\dfrac{107\times 113}{97\times 101}$ while $G(107,12)=\dfrac{109}{97}
< \dfrac{113}{101}\cdot$ 

\item
In \eqref{Gmin1}, we have seen that $\ell(G(p,m))=m$ implies
$G(p,m)\ge 1+\frac{m}{p_k}$, and it turns out that this last
inequality seems to hold for $m$ large enough. 
\end{itemize}

%--------------------------------------------------------------------------
\subsection{A second way to compute $G(p_k,m)$}\label{newalgo}
%--------------------------------------------------------------------------

We want to compute $G(p_k,m)$ for a large $m$.
The following proposition says that if, for some small $\delta$,
$p_k-m + \delta$ is prime and such that $G(p_{k+1},\delta)$
is not too small, 
then the computation of $G(p_{k},m)$ is reduced to the computation
of $G(p_{k+1},m')$ for few small values of $m'$.

\begin{prop}\label{Gnew}
We want to compute $G(p_k,m)$ as defined in \eqref{G1} or \eqref{G}
with $p_k$ odd and $p_{k+1}-p_k\le m \le p_{k+1}-3$. We 
assume that we know some even non-negative integer $\de$ satisfying
\begin{equation}\label{dp}
p_{k+1}+\de-m \quad\text{ is prime, }
\end{equation}
\begin{equation}\label{di}
G(p_{k+1},\de) \ge 1+ \frac{\de}{p_{k+1}}
\end{equation}
and
\begin{equation}\label{dm}
\de < \frac{2m}{9} < \frac{2p_{k+1}}{9}\cdot
\end{equation}
If $\de=0$, we know from Proposition \ref{propgpk} that 
$G(p_k,m)=\dfrac{p_{k+1}}{p_{k+1}-m}\cdot$ If $\de > 0$, we have
\begin{equation}\label{Gq^}
G(p_k,m)=\max_{\substack{q\;\text{ prime}\\p_{k+1}-m\;\le \;q\;\le \;\qh}} \;\;
\frac{p_{k+1}}{q} \;\,G(p_{k+1},m-p_{k+1}+q),
\end{equation}
where $\qh$ is defined by
\begin{equation}\label{qdef}
\qh=\frac{p_{k+1}p_{k+2}(p_{k+1}-m+\de)}{(p_{k+1}+\de)(p_{k+1}-3\de/2)}
\le p_{k+2}-m+\frac{3\de}{2}\cdot
\end{equation}
\end{prop}

Before proving Proposition \ref{Gnew} in Section \ref{secdeg}, we
shall first think to the possibility of applying it to compute
$G(p_k,m)$.

%--------------------------------------------------------------------------
\subsection{Large differences between consecutive primes}\label{diffprem}
%--------------------------------------------------------------------------

For $x\ge 3$, let us define
\begin{equation}\label{Deltapr}
\De(x)=\max_{p_j\le x} (p_j-p_{j-1}).
\end{equation} 
Below, we give some values of $\De(x)$:
\[
\begin{array}{|r|ccccccccccc|}
\hline
x&10^2&10^3&10^4&10^5&10^6&10^7&10^8&10^9&10^{10}&10^{11}&10^{12}\\[1mm]
\hline
\De(x)&8&20&36&72&114&154&220&282&354&464&540\\[1mm]
\hline
(\log x)^2&21&48&85&133&191&260&339&429&530&642&763\\
\hline
\end{array}
\]
A table of $\De(x)$ up to $4\cdot 10^{12}$ calculated by D. Shanks, 
L.J. Lander, T.R. Parkin and R. Brent can be found in \cite{Rie}, p. 85.
There is a longer table (up to $8\cdot 10^{16}$) on the web site 
\cite{Nicely}. H. Cram\'er 
conjectured in \cite{Cra} that $\lim_{x\to\iy} \frac{\De(x)}{(\log x)^2}=1$.
For $x\le 8\cdot 10^{16}$, $\De(x)\le 0.93 (\log x)^2$ holds.

\medskip

Let us set $\De=\De(p_{k+1})$; let us denote by $\de_1=\de_1(p_k)$ the 
smallest even integer such that $\de_1\ge \De$ and 
\begin{equation}\label{m1}
G(p_{k+1},d) \ge 1+\frac{d}{p_{k+1}},\quad  
d=\de_1-\De+2,\de_1-\De+4, \ldots, \de_1.
\end{equation}
By using the combinatorial algorithm described in
\ref{comHG}, we have computed that for all primes $p_k\le 3\cdot 10^8$, we have
$\de_1(p_k)\le 900=\de_1(252314747)$ and
\begin{equation}\label{m1maj}
\de_1(p_k)\le 2.55(\log p_k)^2.
\end{equation}
To compute the suffix of $g(n)$ for $n\le 10^{15}$, we do not have to deal
with larger values of $p_k$.
However, for larger $p_k$'s, we conjecture that $\de_1(p_k)$ exists and 
is not too large.

\begin{lem}\label{lemdel}
Let $p_k$ satisfy $5\le p_k\le 3\cdot 10^8$, $m$ be an even integer such that
$p_{k+1}-p_k\le m\le p_{k+1}-3$, and $\de_1=\de_1(p_k)$ defined by \eqref{m1}.
If $m\ge \frac 92 \de_1(p_k)$, then there exists an even non-negative integer
\begin{equation}\label{dede1}
\de=\de(p_k,m)\le \de_1(p_k)\le 2.55(\log p_k)^2
\end{equation}
such that \eqref{dp}, \eqref{di} and \eqref{dm} hold. Therefore, Proposition
\ref{Gnew} can be applied to compute $G(p_k,m)$.
\end{lem}

\begin{proof}
Let us set $a=p_{k+1}+\de_1(p_k)-m$. We have
\[
a=p_{k+1}+\de_1(p_k)-m
\le p_{k+1}-\frac{7}{2}\de_1(p_k) 
\le p_{k+1}-\frac{7}{2}\Delta
< p_{k+1}. 
\] 
Since $\de_1\ge \De$ and $m\le p_{k+1}-3$, $a\ge \De+3$ holds.
From the definition of 
$\De=\De(p_{k+1})$, there exists an even number $b$, $0\le b\le\De-2$
such that $a-b=p_{k+1}-m+(\de_1-b)$ is prime. From the definition 
of $\de_1(p_k)$,
we know that $G(p_{k+1},\de_1-b) \ge 1+\frac{\de_1-b}{p_{k+1}}$. Therefore, 
$\de=\de_1-b$ satisfies \eqref{dp}, \eqref{di}, \eqref{dm} and
$0\le \de \le \de_1(p_k)$.
The last upper bound of \eqref{dede1} follows from \eqref{m1maj}.
\end{proof}

%--------------------------------------------------------------------------
\subsection{Proof of Proposition \ref{Gnew}}\label{Gproof}
%--------------------------------------------------------------------------
\subsubsection*{A polynomial equation of degree $2$}\label{secdeg}
%--------------------------------------------------------------------------

\begin{lem}\label{eq2}
Let us consider real numbers $T_1,T_2,\de$ satisfying
\begin{equation}\label{Qeq}
0<T_1 < T_2
\end{equation}
and
\begin{equation}\label{deleq}
(\de=0\; \text{ or } \; \de\ge T_2-T_1) \quad \text{  and }\quad 
\de<\frac{2T_1}{9}\cdot
\end{equation}
Note that \eqref{Qeq} and \eqref{deleq} imply 
\begin{equation}\label{taueq}
T_1+\de \le \frac{T_1T_2}{T_2-\de}\cdot
\end{equation}
Let $m$ be a parameter satisfying
\begin{equation}\label{meq}
0\le \frac{9\de}{2} \le m < T_1.
\end{equation}
We set
\begin{equation}\label{Eeq}
E(X)=X^2-(T_1+T_2-m)X+\frac{T_1T_2(T_1+\de-m)}{T_1+\de}\cdot
\end{equation}
\begin{enumerate}
\item
The equation $E(X)=0$ has two roots $X_1$ and $X_2$ satisfying 
\begin{equation}\label{X1X2}
0 < X_1 < \frac{T_1+T_2-m}{2} < X_2 \le T_2-\de.
\end{equation}
\item
 For $T_1,T_2$ and $\de$ fixed and $m$ in the range 
\eqref{meq}, $X_2$ is a non-decreasing function of $m$. 

\item
We have
\begin{equation}\label{X2min}
T_1-\frac{3\de}{2} < \frac{T_1+2T_2}{3}-\frac{3\de}{2}\le  X_2\le T_2-\de.
\end{equation}
\item Let $Y_1$ and $Y_2$ be two positive real numbers satisfying
\begin{equation}\label{Y1eq}
Y_1 < Y_2,\quad Y_1+Y_2=T_1+T_2-m\;\text{ and }\; \frac{T_1T_2}{Y_1Y_2} \ge 
\frac{T_1+\de}{T_1+\de-m}\cdot
\end{equation}
We have
\begin{equation}\label{Y3eq}
Y_2 \ge X_2\ge T_1-\frac{3\de}{2} \qquad \text{ and } \qquad Y_1 \le X_1\le
T_2-m+\frac{3\de}{2}. 
\end{equation}
\end{enumerate}
\end{lem}

\begin{proof}
\begin{enumerate}
\item  The discriminant D of \eqref{Eeq} can be written as
\begin{eqnarray}\label{D2eq}
D &=&(T_1+T_2-m)^2-4\frac{T_1T_2(T_1+\de-m)}{T_1+\de}\notag\\
&=& (m+T_2-T_1)^2\left[1-\frac{4\de}{m}\frac{m^2T_2}{(m+T_2-T_1)^2
    (T_1+\de)} \right], 
\end{eqnarray}
since, from \eqref{Qeq} and \eqref{meq}, $m+T_2-T_1$ does not vanish.
If $\de=0$, the above bracket is $1$ while if $\de\ge T_2-T_1 > 0$, 
the fractions  $\dfrac{T_2}{T_1+\de}$ and 
$\dfrac{m}{m+T_2-T_1}$ are at most $1$, so that in both cases
\eqref{D2eq} yields 
\begin{equation}\label{D3eq}
D \ge (m+T_2-T_1)^2\left[1-\frac{4\de}{m} \right].
\end{equation}
Therefore, from \eqref{meq} and \eqref{Qeq}, 
$D \ge \dfrac{(m+T_2-T_1)^2}{9} > 0$ holds. 

The sum $X_1+X_2$ of the two roots is $T_1+T_2-m$ which explains the
second and the third inequality of \eqref{X1X2}. 
Further, since $T_1 < T_2$ and $m\ge 2\de$,
$\dfrac{T_1+T_2-m}{2}\le T_2-\de$ holds. By \eqref{Eeq}, 
\eqref{meq} and \eqref{taueq},
\[E(T_2-\de)=(T_1+\de-m)\left(\frac{T_1T_2}{T_1+\de}-(T_2-\de)\right) \ge 0\]
which proves the last inequality of \eqref{X1X2}.

\begin{rem}\label{remX2}
If $\de=0$, the roots of \eqref{Eeq} are $X_1=T_1-m$ and $X_2=T_2$.
If $\de=T_2-T_1$, they are $X_1=T_2-m$ and $X_2=T_1$. 
\end{rem}

\item By \eqref{Eeq}, $X_2$ is implicitely defined in terms of $m$
and, through \eqref{taueq}, we have
\[\frac{{d}\,X_2}{{d}\,m}=\frac{-\pd Em}{\pd EX}=
\frac{\frac{T_1T_2}{T_1+\de}-X_2}{2X_2-(T_1+T_2-m)}\ge \frac{T_2-\de-X_2}
{2X_2-(T_1+T_2-m)}\]
which is non-negative from \eqref{X1X2}.

\item
 For $m=\frac{9\de}{2}$, \eqref{D3eq} 
yields $\sqrt D \ge \frac{m+T_2-T_1}{3}=\frac{3\de}{2}+\frac{T_2-T_1}{3}$ and 
\[X_2=\frac{T_1+T_2-m+\sqrt D}{2} \ge \frac{T_1+2T_2}{3}-\frac{3\de}{2}
\ge T_1-\frac{3\de}{2}\cdot\] 
Further, for  $m\ge \frac{9\de}{2}$, the upper bound in
\eqref{X2min} follows from (ii).

\item  
Conditions \eqref{Y1eq} imply 
$E(Y_1)=E(Y_2)=-Y_1Y_2+\dfrac{T_1T_2(T_1+\de-m)}{T_1+\de}\ge 0$
so that $Y_1\le X_1$ and $Y_2\ge X_2$; \eqref{Y3eq} follows from \eqref{X2min}
and from $X_1=T_1+T_2-m-X_2$.
\end{enumerate}
\end{proof}

%--------------------------------------------------------------------------
\subsubsection*{Structure of the fraction $G(p_k,m)$}\label{GTAU}
%--------------------------------------------------------------------------

\begin{lem}\label{Gtau}
Let $k$ and $m$ be  integers such that $k\ge 3$ and 
$p_{k+1}-p_k\le m \le  p_{k+1} -3$. 
We write
\begin{equation}\label{F}
G(p_k,m) = F=\frac{Q_1Q_2\ldots Q_s}{q_1q_2\ldots q_s}
\end{equation}
with $s\ge 1$ and $Q_1,\ldots, Q_s,q_1,\ldots,q_s$ primes satisfying
\begin{equation}\label{Qqs}
3\le q_s <q_{s-1} < \ldots < q_1 \le p_k <p_{k+1} \le  
Q_1 < \ldots Q_{s-1} < Q_s,
\end{equation}
\begin{equation}\label{lF}
p_{k+1}-p_k\le \ell(F)=\sum_{i=1}^s (Q_i-q_i)\le m \le p_{k+1}-3 <p_{k+1}
\end{equation}
and we assume that there exists an integer $\de$ such that
\begin{equation}\label{delta}
0\le \de < \frac{2m}{9},\quad \text{ and }\quad 
(\de=0 \text{ or } \de\ge p_{k+2}-p_{k+1}) 
\end{equation}
and
\begin{equation}\label{Ftau}
F\ge \frac{p_{k+1}+\de}{p_{k+1}-m+\de}\cdot
\end{equation}
We apply Lemma \ref{eq2} with $T_1=p_{k+1}$ and $T_2=p_{k+2}$, 
$\de$ and $m$, and we denote by $X_1$ and $X_2$ the two roots of 
equation \eqref{Eeq}, $E(X)=0$. Then we have
\[
Q_s \le p_{k+1}+\de, \leqno{1.}
\]
\[\text{ for } s\ge 2 \text{ and } 1\le i \le s-1, 
\qquad\la_i\stackrel{def}{=\!=} Q_i-q_i 
\le p_{k+2}-X_2, \leqno{2.}\]
\[\text{ for } s\ge 2 \text{ and } 1\le j \le s-1, 
\qquad \La_j \stackrel{def}{=\!=}\sum_{i=1}^j \la_i 
\le p_{k+2}-X_2. \leqno{3.}\]
Moreover, if we write $F=UV$ with
\begin{equation}\label{UV}
U=\frac{Q_1Q_2\ldots Q_{s-1}Q_s}{q_1q_2\ldots q_{s-1}p_{k+1}}\qquad 
\text{ and } \qquad V=\frac{p_{k+1}}{q_s},
\end{equation}
we have, for $s\ge 1$
\[
\ell(U)=\La_{s-1}+Q_s-p_{k+1} \le p_{k+2}-X_2
\le p_{k+2}-p_{k+1}+\frac{3\de}{2}\leqno{4.}
\]
and
\[p_{k+1}-m\le q_s \le \qh= 
\frac{p_{k+1}p_{k+2}(p_{k+1}-m+\de)}
{(p_{k+1}+\de)(p_{k+1}-3\de/2)}\cdot \leqno{5.}\]
\end{lem}

\begin{proof}
\begin{enumerate}
\item
First, we observe that \eqref{Qqs} implies
\begin{equation}\label{Qi}
Q_i\ge p_{k+i}\ge p_{k+1},\qquad 1\le i\le s.
\end{equation}

Lemma \ref{lemme1} and \eqref{lF}
yield  respectively $F\le \frac{Q_s}{Q_s-\ell(F)}$
and $\ell(F)\le m$, so that, 
together with \eqref{Ftau}, we get
\[\frac{p_{k+1}+\de}{p_{k+1}+\de-m}\le F \le
\frac{Q_s}{Q_s-\ell(F)}\le \frac{Q_s}{Q_s-m}\]
which, with the decrease of $t \mapsto \frac{t}{t-m}$,
gives $Q_s\le p_{k+1}+\de$.

\item
 From the definition of $\la_i$ and \eqref{Qqs}, $\la_i$ is
positive and increasing on $i$, and it suffices to show $\la_{s-1}\le
p_{k+2}-X_2$.  We write $F=F_1F_2$ with $F_1=\frac{Q_{s-1}}{q_{s-1}}$
and $F_2=\prod_{i\ne s-1} \frac{Q_i}{q_i}$.  From \eqref{lF} and
\eqref{Qqs}, we have
\begin{equation*}
p_{k+1} > m > m-\la_{s-1}\ge \ell(F)-\la_{s-1}=\la_1+\ldots+\la_{s-2}+
\la_s\ge \la_s> \la_{s-1}
\end{equation*}
which implies
\begin{equation}\label{sm1}
p_{k+2}-\la_{s-1}  > p_{k+1}-\la_{s-1} > p_{k+1}-(m-\la_{s-1}).
\end{equation}
Further, Lemma \ref{lemme1}, 
\eqref{lF} and the increase of $t \mapsto \frac{Q_s}{Q_s-t}$,
\eqref{Qi} and the decrease of $t \mapsto \frac{t}{t-(m-\la_{s-1})}$,
imply
\begin{eqnarray}\label{F2}
F_2 &\le&  \frac{Q_s}{Q_s-\ell(F_2)}=
\frac{Q_s}{Q_s-(\ell(F)-\la_{s-1})}\notag\\
&\le& \frac{Q_s}{Q_s-(m-\la_{s-1})}\le
\frac{p_{k+1}}{p_{k+1}-(m-\la_{s-1})}\cdot
\end{eqnarray}

\ni
If $s\ge 3$ or $Q_1\ge p_{k+2}$, \eqref{Qqs} implies 
$Q_{s-1}\ge p_{k+2}$ which yields
$F_1=\frac{Q_{s-1}}{Q_{s-1}-\la_{s-1}}\le \frac{p_{k+2}}{p_{k+2}-\la_{s-1}}$
so that, from \eqref{Ftau} and \eqref{F2}, we get
\begin{equation}\label{Fmaj}
\frac{p_{k+1}+\de}{p_{k+1}+\de-m}\le F= F_1F_2 \le 
\frac{p_{k+2}}{p_{k+2}-\la_{s-1}}\;
\frac{p_{k+1}}{p_{k+1}-(m-\la_{s-1})}\cdot
\end{equation}
Let us set $Y_2=p_{k+2}-\la_{s-1}$, $Y_1=p_{k+1}-(m-\la_{s-1})$;
from \eqref{sm1}, $Y_2 > Y_1$ holds and, 
in view of \eqref{Fmaj}, we may apply Lemma \ref{eq2}, Point 4. to get
$Y_2=p_{k+2}-\la_{s-1} \ge X_2$ which implies 2..
\medskip

\ni
If $s=2$ and $Q_1=p_{k+1}$, 
$F=\frac{p_{k+1}}{q_1}\frac{Q_2}{q_2} = \frac{p_{k+1}}{q_1}
\frac{Q_2}{Q_2-(Q_2-q_2)}\cdot$
From \eqref{Qi} we have $Q_2\ge p_{k+2}$ and
$F\le \frac{p_{k+1}}{q_1}\frac{p_{k+2}}{p_{k+2}-(Q_2-q_2)}$. 
Here we set $Y_2=q_1$ and 
$Y_1=p_{k+2}-(Q_2-q_2)=q_2-(Q_2-p_{k+2})$; 
by \eqref{Qqs} and \eqref{lF}, we get 
\begin{eqnarray*}
Y_2=q_1>q_2 \ge Y_1 \hspace{-2mm} &=&\hspace{-2mm} 
q_2-(Q_2-p_{k+2})=p_{k+2}-\la_2 \\
&\ge& \hspace{-2mm} p_{k+2}-\sum_{i=1}^2 \la_i=p_{k+2}-\ell(F)
\ge p_{k+2}-m > 0;
\end{eqnarray*}
we may still apply Lemma \ref{eq2} Point 4. to get
$Y_2=q_1=p_{k+1}-\la_1\ge X_2$, which implies 2..

\item
 This time, we write $F=F_1F_2$ with $F_1=\prod_{i=1}^j \frac{Q_i}{q_i}$
and $F_2=\prod_{i=j+1}^s \frac{Q_i}{q_i}$ so that $\ell(F_1)=\La_j$
and $\ell(F_2)=\ell(F)-\La_j\le m-\La_j$. 
For $2\le j\le s-1$,
from \eqref{Ftau}, Lemma \ref{lemme1}, \eqref{Qi},
and \eqref{lF} we get
\begin{eqnarray*}%\label{Fmaj}
\frac{p_{k+1}+\de}{p_{k+1}+\de-m}\le F=F_1F_2 &\le& \frac{Q_j}{Q_j-\ell(F_1)}\;
\frac{Q_s}{Q_s-\ell(F_2)}\\
&\le& \frac{p_{k+2}}{p_{k+2}-\La_j}\;
\frac{p_{k+1}}{p_{k+1}-(m-\La_{j})}\cdot
\end{eqnarray*}
Therefore, we apply Lemma \ref{eq2} Point 4., but we do not know whether
$p_{k+2}-\La_j$ is greater than $p_{k+1}-(m-\La_{j})$, so that, either
\begin{equation}\label{cas1}
p_{k+2}-\La_j\ge X_2
\end{equation}
or
\begin{equation}\label{cas2}
p_{k+2}-\La_j\le X_1.
\end{equation}
For $j=1$, as $\La_1=\la_1$, \eqref{cas1} holds, from 2.. Since $\La_j$ is 
increasing on $j$, if \eqref{cas1} holds for some $j=j_0$, it also holds for 
$j\le j_0$. If \eqref{cas1} holds for $j=s-1$, 3. is proved; so, let 
us assume that the greatest value $j_0$ for which \eqref{cas1} holds  
satisfies  $1\le j_0 < s-1$; we should have
\begin{equation}\label{casj0}
p_{k+2}-\La_{j_0}\ge X_2\qquad \text{ and } \qquad p_{k+2}-\La_{j_0+1}\le X_1.
\end{equation}
From 2., \eqref{casj0} and because $X_1,X_2$ are solutions of
\eqref{Eeq}, we should get
\[p_{k+2}-X_2\ge \la_{j_0+1}=\La_{j_0+1}-\La_{j_0}\ge X_2-X_1=
2X_2+m-p_{k+1}-p_{k+2}\]
which, would imply $m\le 2p_{k+2}+p_{k+1}-3X_2$ and, through the
second inequality of \eqref{X2min},
$m \le \frac{9\de}{2}$, in   
contradiction with \eqref{delta}. Therefore, $j_0\ge s-1$ and 3. is proved.

\item
 If $s=1$ we have to show $\ell(U)=Q_1-p_{k+1}\le p_{k+2}-X_2$ 
which is true since, from 1., $Q_1-p_{k+1}\le \de$ and from \eqref{X2min}, 
with $T_2 = p_{k+2}$,  $\de \le p_{k+2}-X_2$.

So, we assume $s\ge 2$. If $Q_1=p_{k+1}$, $U$ simplifies itself;
and, in all cases, from \eqref{Qqs}, the prime factors of the 
numerator of $U$ are 
at least $p_{k+2}$ and those of the denominator are at most $p_{k+1}$. 
So, we may apply Lemma \ref{lemme1} which, with \eqref{Qi} and  
the decrease of $t \mapsto t/(t-\ell(U)$, yields
\begin{equation}\label{U1}
U\le \frac{Q_s}{Q_s-\ell(U)}\le \frac{p_{k+2}}{p_{k+2}-\ell(U)},\qquad
V=\frac{p_{k+1}}{p_{k+1}-\ell(V)}\cdot
\end{equation}
It follows from \eqref{lF} that $\ell(U)+\ell(V)=\ell(F) \le m$ and, 
from \eqref{Ftau}, we get
\begin{equation*}
\frac{p_{k+1}+\de}{p_{k+1}+\de-m}\le F=UV
\le \frac{p_{k+2}p_{k+1}}{(p_{k+2}-\ell(U))(p_{k+1}-(m-\ell(U)))}\cdot
\end{equation*}
Applying Lemma \ref{eq2} Point 4. with 
$(Y_1,Y_2) = (p_{k+2}-\ell(U),p_{k+1}-(m-\ell(U)))$
yields 
\begin{equation}\label{cas3}
p_{k+2}-\ell(U)\ge X_2 \qquad \text{ or } \qquad p_{k+2}-\ell(U)\le X_1.
\end{equation}
But, from 1. and 3., we have $\ell(U)=\La_{s-1}+Q_s-p_{k+1}
\le p_{k+2}-X_2+\de$ which, together with $(X_1,X_2)$ solutions
of \eqref{Eeq}, the second inequality in \eqref{X2min} 
and \eqref{delta}, give
\begin{eqnarray*}
X_1+\ell(U)-p_{k+2} &\le& X_1-X_2+\de=\de+p_{k+1}+p_{k+2}-m-2X_2\\
&\le& \de+p_{k+1}+p_{k+2}-m-\frac 23(p_{k+1}+2p_{k+2})+3\de\\ 
&=& 4\de+\frac{p_{k+1}-p_{k+2}}{3}-m < 0.
\end{eqnarray*}
Therefore, $p_{k+2}-\ell(U)\le X_1$ does not hold, 
and, from \eqref{cas3}, we have $p_{k+2}-\ell(U)\ge X_2$ which shows
the first inequality in 4.. The second inequality comes from \eqref{X2min}. 

\item From \eqref{Qqs} and \eqref{lF}, we have 
$\ell(V)=p_{k+1}-q_s\le Q_s-q_s\le \ell(F) \le m$ which 
proves the lower bound of 5.. 

If $s=1$ and $Q_1=p_{k+1}$, $U=1$ and $F=V$ so that, from \eqref{Ftau},
\[q_s=\frac{p_{k+1}}{F}\le \frac{p_{k+1}(p_{k+1}-m+\de)}{p_{k+1}+\de}
\le \qh =\frac{p_{k+1}p_{k+2}(p_{k+1}-m+\de)}
{(p_{k+1}+\de)(p_{k+1}-3\de/2)}\cdot
\]

If $s\ge 2$ or $Q_1\ge p_{k+2}$, \eqref{U1} holds and gives with \eqref{Ftau}
and 4.
\[q_s=\frac{p_{k+1}}{V}=\frac{p_{k+1}U}{F}
\le \frac{p_{k+1}p_{k+2}(p_{k+1}-m+\de)}{(p_{k+1}+\de)(p_{k+2}-\ell(U))}
\le \qh.\]
\end{enumerate}
\end{proof}

%--------------------------------------------------------------------------
\subsubsection*{Proof of Proposition \ref{Gnew}}
%--------------------------------------------------------------------------

Let us assume $\de > 0$. \eqref{di} and \eqref{Gpkm1} imply
\begin{equation}\label{del12}
\de\ge p_{k+2}-p_{k+1}.
\end{equation}
First, we prove the upper bound \eqref{qdef}.
We have to show that the quantity below is positive:
\[
(p_{k+2}-m+\de)(p_{k+1}+\de)\Big(p_{k+1}-\frac{3\de}{2}\Big)
- p_{k+1}p_{k+2}(p_{k+1}-m+\de).
\]
But this quantity is equal to
\begin{multline*}
(p_{k+2}-p_{k+1})\Big((p_{k+1}-\de)(m-\frac{3\de}{2})+\de(m-3\de)\Big)\\
+ p_{k+1}\frac{\de}{2}\Big(m-\frac{9\de}{2}\Big)
+\frac{3\de^2}{4}\Big(m-\frac{3\de}{2}\Big)
\end{multline*}
which is clearly positive since, from \eqref{dm}, $p_{k+1} > m >
\frac{9\de}{2}$ holds and \eqref{qdef} is proved.
\medskip

Let $q$ be a prime satisfying $p_{k+1}-m\le q\le \qh$. In view of proving
\eqref{Gq^}, let us show that
\begin{equation}\label{p4min}
\frac{p_{k+1}}{q}G(p_{k+1},m-p_{k+1}+q)\le G(p_k,m)
\end{equation}
holds. Let $q'$ be any prime dividing the denominator of 
$G(p_{k+1},m-p_{k+1}+q)$; we should have $p_{k+2}-q'\le m-p_{k+1}+q$
i.e., $q'\ge p_{k+1}+p_{k+2}-m-q$ 
which yields from \eqref{qdef}, \eqref{del12} and \eqref{dm}
\begin{eqnarray*}
q'-q &\ge& p_{k+1}+p_{k+2}-m-2q\ge p_{k+1}+p_{k+2}-m-2\qh\\
     &\ge&p_{k+1}+p_{k+2}-m-2\left(p_{k+2}-m+\frac{3\de}{2}\right)
       =p_{k+1}-p_{k+2}+m-3\de\\
     &\ge& p_{k+1}-(\de+p_{k+1})+m-3\de=m-4\de > 0.  
\end{eqnarray*}
Therefore, $q'\ne q$, and after a possible simplification by $p_{k+1}$,
$\frac{p_{k+1}}{q}G(p_{k+1},m-p_{k+1}+q)\in \cg(p_k,m)$ (defined 
in \eqref{cG}), which, from \eqref{G}, implies \eqref{p4min}.

From \eqref{del12} and \eqref{dm}, we have $0 < 2\de < m$, and the prime
$p=p_{k+1}+\de-m$ satisfies $p < p_{k+2}-\de$, and thus
is smaller than any prime factor of the denominator of $G(p_{k+1},\de)$.
Therefore, after possibly simplifying by $p_{k+1}$, the fraction
$\Phi=\frac{p_{k+1}}{p} G(p_{k+1},\de)$ belongs to $\cg(p_k,m)$ 
and we have from \eqref{G} and \eqref{di}
\[G(p_k,m)\ge \Phi\ge \frac{p_{k+1}}{p_{k+1}+\de-m}
\left(1+\frac{\de}{p_{k+1}}\right)=\frac{p_{k+1}+\de}{p_{k+1}+\de-m}\cdot\]
So, hypotheses \eqref{delta} and \eqref{Ftau} being fullfilled,
we may apply Lemma \ref{Gtau}, (v) which, under the notation \eqref{UV}, 
asserts that
\begin{equation}\label{GUV}
G(p_k,m)=UV=U\frac{p_{k+1}}{q_s}
\end{equation}
with $q_s\in [\,p_{k+1}-m,\qh\,]$ and $\ell(U)+\ell(V)=\ell(G(p_k,m))$
which, from \eqref{lGm},  implies $\ell(U)\le m-\ell(V)=m-p_{k+1}+q_s$.
After a possible simplification by $p_{k+1}$, $U$ belongs to 
$\cg(p_{k+1},\ell(U)) \subset \cg(p_{k+1},m-p_{k+1}+q_s)$.
So, from \eqref{G}, $U\le G(p_{k+1},m-p_{k+1}+q_s)$, and \eqref{GUV} gives
\[G(p_k,m)\le \frac{p_{k+1}}{q_s}G(p_{k+1},m-p_{k+1}+q_s)\]  
which, with \eqref{p4min}, completes the proof of \eqref{Gq^} and of 
Proposition \ref{Gnew}.   \hfill$\Box$

%--------------------------------------------------------------------------
\section{Some results}\label{perfnewalgo}
%--------------------------------------------------------------------------

With the {\sc maple} program available on the web-site of J.-L. Nicolas,
the factorization of $g(n)$ has been computed 
for some values of $n$. The results for $n=10^6,10^{9},
10^{12},10^{15}$ are displayed in {\sc Fig}. 6.
For primes $q_1 < q_2$ let us denote by $[q_1\!-\!q_2]$ the product 
$\prod_{q_1 \le p \le q_2} p$.
The bold factors in the values of $g(n)$ are the factors
of the plain prefix $\pi$ of $g(n)$, defined in \eqref{defsuf}.
\begin{figure}[t]
\[
\begin{array}{|l|}
\hline\\%%[1ex]
n=10^6,\hfill N=2^9 3^6 5^4 7^3 [11\!-\!41]^2 [43\!-\!3923]
\\[2ex]
\ell(N)=998093,\hfill
g(10^6)=g(10^6-1)=\dfrac{\mathbf{43} \cdot3947}{3847}N.
\\[2.5ex]
\hline\\[2ex]
n=10^9,\hfill N=2^{14}3^9 5^6 7^5 11^4 13^4 
[17\!-\!31]^3 [37\!-\!263]^2 [269\!-\!150989]
\\[2ex]
\ell(N)=999969437,\hfill
g(10^9)=g(10^9-1)=\dfrac{\mathbf{37}\cdot 150991}{\mathbf{2\cdot 3} \cdot
  148399}N.
\\[2.5ex]
\hline\\[2ex]
n=10^{12},%%\\[3ex]
N = 2^{18}3^{12} 5^8 7^6 11^5 13^5 [17\!-\!31]^4  [37\!-\!113]^3
[127\!-\!1613]^2 [1619\!-\!5476469]
\\[2ex]
\ell(N)=999997526071,\hfill
g(10^{12})=\dfrac{\mathbf{1621\cdot 1627\cdot 1637}\cdot 5476483}
{5475739\cdot 5476469}N.
\\[2.5ex]
\hline\\[2ex]
n=10^{15},\qquad N = 2^{23}3^{15} 5^{10} 7^8 11^7 13^6 17^6 
[19\!-\!31]^5  [37\!-\!79]^4 [83\!-\!389]^3\\[1ex]
\hfill
\times [397\!-\!9623]^2 [9629\!-\!192678817],
\\[2ex]
\ell(N)=999999940824564,
\\[2ex]
\hfill
g(10^{15})=g(10^{15}-1)=
\dfrac{192678823 \cdot 192678853 \cdot 192678883 \cdot 192678917}
{\mathbf{389\cdot 9539\cdot 9587\cdot 9601\cdot 9619\cdot 9623}\cdot 192665881} 
N.
\\[2ex]
\hline
\end{array}
\]
\caption{The values $g(n)$ for $n = 10^{6}, 10^{9}, 10^{12}, 10^{15}$.}
\end{figure}

On a 3GHz Pentium 4, the time of computation of $g(n)$ is 
about 0.02 second for an
integer $n$ of 6 decimal digits and 10 seconds for 15 digits.
\nopagebreak[4]
%--------------------------------------------------------------------------
\section{Open problems}\label{openpb}
%--------------------------------------------------------------------------
\subsection{An effective bound for the benefit}\label{openben}
%--------------------------------------------------------------------------
Let us define $\ben g(n)$ by \eqref{ben} with $N$ and $\rho$ defined by
\eqref{NN'} and \eqref{Nxj}. Is it possible to get 
an effective form of \eqref{benmaxg}, i.e.,
\[\ben g(n)+n-\ell(g(n)) \le C\rho\]
for some absolute constant $C$ to determine?

A hint is to apply Proposition \ref{propmajben} with $M=\frac{P_1P_x
2\ldots P_r}
{q_1q_2\ldots q_{2r}}$ for some $r$, where the $P_i$'s are the $r$ 
smallest primes not dividing $N$ and the $q_i$'s are the $2r$ largest primes
such that $v_{q_i}(N) =2$, and, further, to apply effective results
on the Prime Number Theorem like those of \cite{RS} or \cite{Dus}.

%--------------------------------------------------------------------------
\subsection{Increasing subsequences of $g(n)$}
%--------------------------------------------------------------------------

An increasing subsequence of $g$ is a set of $k$ consecutive 
integers $\{n,n+1,\ldots,n+k-1\}$ such that
\begin{equation}\label{incr}
g(n-1)=g(n) < g(n+1) < \ldots < g(n+k-1) = g(n+k).
\end{equation}
Due to a parity phenomenom, these maximal sequences are rare. For $n\le 10^6$,
there are only $9$ values on $n$ with $k\ge 7$. The record is $n=35464$ 
with $k=20$.

Are there arbitrarily long maximal sequences? It seems to be a very 
difficult question.
In \cite{N60}, (1.7), it is conjectured that there are infinitely many maximal 
sequences with $k\ge 2$.

%----------------------------------------------------------------------------
\subsection{The second minimum}
%----------------------------------------------------------------------------

Let us write $g_1(n)=g(n) > g_2(n) > \ldots > g_I(n)=1$ all the integers 
such that, if $\si\in {\mathfrak S}_n$, the order of $\si$ is equal to 
$g_i(n)$ for some $i\in \{1,2,\ldots,I\}$. From \eqref{ErTu}, $I$ is equal 
to the number of positive integers $M$ satisfying $\ell(M)\le n$.

We might be interested in the computation of $g_2(n)$ or more generally, in 
the computation of $g_i(n)$ for $1\le i \le i_0$ where $i_0$ is some (small)
fixed constant.

The basic algorithm (see Section \ref{basic}) can be easily adapted for 
this purpose. It seems reasonnable to think that our algorithm, as sketched in
\ref{stepalgo}, can also be extended to get $g_i(n)$.

%----------------------------------------------------------------------------
\subsection{Computing $h(n)$}
%----------------------------------------------------------------------------

Let $h(n)$ be the maximal product of primes $p_{i_1},p_{i_2},\ldots,p_{i_r}$
under the condition $p_{i_1}+p_{i_2}+\ldots +p_{i_r}\le n$ ($r$ is not fixed);
$h(n)$ can be interpreted as the maximal order of a permutation of the 
symmetric group ${\mathfrak S}_n$ such that the lengths of its cycles are all 
primes.

A formula similar to \eqref{g} can be written:
\[h(n)=\max_{\substack{M\;\text{ squarefree }\\ \ell(M)\le n}} M.\] 
The superchampion numbers are the product of the first primes.

A related problem is to find an algorithm to compute $h(n)$ for $n$ up
to $10^{15}$.

%--------------------------------------------------------------------------
\subsection{Maximum order in $GL(n,\Z)$}
%--------------------------------------------------------------------------

Let $G(n)$ be the maximum order of torsion elements in $GL(n,\Z)$. It 
has been shown in \cite{Lev} that
\begin{equation}\label{GLev}
G(n)=\max_{L(M)\le n}\; M
\end{equation}
where $L$ is the additive function defined by $L(1)=L(2)=0$ and 
$L(p^\al)=\vfi(p^\al)=p^\al-p^{\al-1}$ if $p^\al\ge 3$.

From \eqref{GLev} and \eqref{g}, it follows that $g(n)\le G(n)$ holds for 
all $n$'s and it has been shown in \cite{N73} that 
$\lim_{n\to\iy} G(n)/g(n)=\iy$.

Is it possible to adapt the algorithm described in this paper to compute $G(n)$
up to $10^{15}$?

\def\refname{References}

\bigskip

\begin{minipage}[t]{6.5cm}
Marc Del\'eglise, Jean-Louis Nicolas,\\
Universit\'e de Lyon, \\
Universit\'e de Lyon 1, CNRS,\\
Institut Camille Jordan,\\
B\^at. Doyen Jean Braconnier,\\
21 Avenue Claude Bernard,\\
F-69622 Villeurbanne cedex, France.\\
\end{minipage}
\begin{minipage}[t]{5cm}
Paul Zimmermann,\\
Centre de Recherche INRIA \hfill\\
Nancy Grand Est \hfill\\
Projet CACAO\,-b\^atiment A\\
615 rue du Jardin Botanique,\\
F-54602 Villers-l\`es-Nancy cedex,\\
France.
\end{minipage}
\bigskip

\noindent
e-mails and web-sites:  
\medskip

\noindent
\hspace{-1.em}
\begin{tabular}{ll}
\path{deleglis@math.univ-lyon1.fr} 
& \path{http://math.univ-lyon1.fr/~deleglis}\\
\path{jlnicola@in2p3.fr} 
& \path{http://math.univ-lyon1.fr/~nicolas/}\\
\path{Paul.Zimmermann@loria.fr}
& \path{http://www.loria.fr/~zimmerma/}      
\end{tabular}

\end{document}